\documentclass[10pt]{amsart}
\usepackage[utf8]{inputenc}
\usepackage[english]{babel}
\usepackage{amsmath}
\usepackage{amsfonts}
\usepackage{amssymb}
\usepackage{amsthm}
\usepackage{siunitx}
\usepackage{bbm}
\usepackage[colorlinks,allcolors=blue]{hyperref}
\usepackage[square,numbers]{natbib}
\usepackage{comment}
\usepackage{tasks}
\usepackage{enumitem}
\usepackage{stix}
\usepackage{multicol}
\usepackage{mathtools}
\usepackage{graphicx}
\usepackage{geometry}
\usepackage{mathrsfs}
\usepackage{pgf,tikz,pgfplots}
\newtheorem{lemma}{Lemma}
\newtheorem{theorem}{Theorem}
\newtheorem{proposition}{Proposition}
\newtheorem{definition}{Definition}
\newtheorem{corollary}{Corollary}
\newtheorem{example}{Example}
\newtheorem{remark}{Remark}

\usetikzlibrary{calc}
\usetikzlibrary{arrows}
\pgfplotsset{compat=1.15}

\title{Improved Bounds For Some Asymptotic Formulas for Counting Words in Shift Spaces}

\author{Hamid Naderiyan\\  University of California - Riverside \\
\href{mailto:hnaderiy@ucr.edu}{hnaderiy@ucr.edu}}

\begin{document}

\maketitle

\begin{abstract}
    This paper studies a version of the counting problem in dynamical systems that is of interest, especially in conformal dynamical systems where the functions of the systems are angle preserving. Recently, M. Pollicott and M. Urba\'{n}ski published a result in this context for D-generic systems where the complex transfer operator behaves nicely on the critical line of the Poincar\'{e} series. Their result contains an asymptotic formula for the Apollonian circle packing. We lift the D-generic condition and conformality of the functions system in this paper to see how their asymptotic formula changes. We use some recent Tauberian theorem to show that the formula gets a form whose limit infimum and limit supremum bounds can be obtained in the sharpest sense. Further, we observed an asymptotic of length closely related to this counting problem. In fact, not only the number of words is subject to some formula, but also their length as well.\\
\end{abstract}

\tableofcontents
\begin{large}

\section{Introduction}
The counting problem in math has a long history dating as far back as the Gauss circle problem. Gauss tried to obtain an asymptotic formula for the number of points in the plane with integer coordinates inside a circle of radius $T$ as $T$ grows. Later on, Sierpinski, Walfisz, Iwaniec \& Mozzochi \cite{Iwaniec}, Huxley \cite{Huxley}, Hardy \cite[~ p. 372]{Ivic}, Landau and Hafner \cite{Hafner} contributed to problems closely related to the Gauss circle problem to obtain better estimates. \\
The analogous problem in the context of hyperbolic spaces as well gained a lot of attention starting in 1942 with (unnoticed) work of Delsarte, where he considered the hyperbolic plane $\mathbb{H}^2$ and instead of $\mathbb{Z}^2$ he considered orbit of a point $z \in \mathbb{H}^2$ under the action of a Fuchsian group $G \subseteq \text{PSL}(2,\mathbb{R})$. He obtained an asymptotic formula for the number of $g \in G$ that moves $z$ at most by $T$ as $T$ grows. Here the distance is measured by a hyperbolic metric of constant negative curvature \cite{Delsarte}. Independently, Huber published his result on this problem in 1956. His approach uses spectral decomposition of the Laplacian operator, where $G$ doesn't contain parabolic elements because he assumes the fundamental domain is compact \cite{HUBER1959}. In the same year, Selberg extended this decomposition for the case $G$ contains parabolic elements where the fundamental domain has a finite area. He used the celebrated trace formula for this \cite[~ p. 77]{Selberg1}. This helped Patterson to approach the problem in generality providing some error terms as well \cite{Patterson}. Along with these works, Margulis answered a similar question in higher dimensional hyperbolic space in 1969 \cite[~ p. 48]{Margulis}. Several others have contributed to this problem in different contexts including Sarnak \cite{sarnak}, Lax \& Phillips \cite{Lax}, Parry \& Pollicott \cite{PaPo}, Lalley \cite{lalley}, Mirzakhani \cite{Mirzakhani} and etc.\\
Recently, Pollicott \& Urba\'{n}ski jointly obtained an asymptotic formula in the context of conformal dynamical systems, see corollary \ref{llll} or \cite[~ p. 39]{UP}. For this, it is enough to have graph directed Markov system in which our functions in the system are contractions and satisfy certain properties. The most important one is the conformal property which is angle preserving orientation preserving or orientation reversing. They use the infinite theory of graph directed Markov system developed by Mauldin \& Urba\'{n}ski \cite{MU} and complex transfer operator developed by Pollicott \cite{P-RPF} to obtain an asymptotic formula for counting finite words in the shift space for which the corresponding composition function of the system has derivative at least $e^{-T}$ as $T$ grows. Further, they introduce a slightly different system in which finitely many parabolic elements are allowed and they apply the aforementioned asymptotic formula for this system. These two kinds of conformal systems have many applications one of which is an asymptotic formula for the planer Apollonian circle packing problem. The circle packing was studied in the 1970s by Boyd \cite{Boyd1} and estimates of the number of circles of radius at $1/T$ were obtained by him in the 1980s \cite{Boyd2}. This estimate had major improvement due to Kontorovich \& Oh in 2011 \cite{Kontorovich} and Oh \& Shah in 2012 \cite{Shah}. The former article focuses on two cases: (a) the number of circles of radius at least $1/T$ inside the biggest circle tangent to the three circles that generate the gasket, and (b) the number of circles of radius at least $1/T$ between two parallel lines generating the gasket up to a period of the gasket. One year later, the latter article obtains a similar formula for case (c) number of circles of radius at least $1/T$ bounded in a curvilinear triangle whose sides are parts of three circles tangent to each other. The method for Kontorovich-Oh-Shah is equidistribution of expanding closed horospheres on hyperbolic 3-manifolds $G \symbol{92} \mathbb{H}^3$ where $G$ is a geometrically finite torsion-free discrete subgroup of $\text{PSL}(2,\mathbb{C})$. Further, they use Patterson-Sullivan theory of conformal density (measure) in which the Laplacian operator has simple isolated eigenvalue $-\delta_{G}(2-\delta_{G})$ where $\delta_{G}$ is the Hausdorff dimension of the limit set under the assumption $\delta_{G}>1$ \cite[~ p. 195]{Sullivan}, \cite[~ p. 272]{Pattersonn}.\\
In Pollicott \& Urba\'{n}ski's work the spectral theory is analyzed for the transfer operator instead, where they assume their system has D-generic property which prevents the situation that the transfer operator admitting $1$ as the spectral value on the critical line $\text{Re}(s)=\delta$ of the Poincar\'{e} series except at the exponent itself $s=\delta$. The other condition they impose on the system is strong regularity which can be perceived to be analogous to the assumption $\delta_{G}>1$  mentioned above.\\
In this paper, we relax the D-generic assumption to see how Pollicott \& Urba\'{n}ski's result changes, see theorem \ref{main thm}. We noticed that in this situation, we no longer obtain only one asymptotic formula. We may obtain continuum many relations. More precisely we can see that the ratio can converge to a full range of a closed interval rather than just a point in Pollicott \& Urba\'{n}ski's result, see example \ref{sharpness}. However, we can obtain a lower bound for the infimum and an upper bound for the supremum. These bounds are shown to be sharp by an example, see example \ref{sharpness}. We should mention that we only assume we are given a real-valued summable H\"{o}lder-type function on the shift space. We don't assume necessarily the function is induced by a conformal system. The main result (theorem \ref{main thm}) involves spectral analysis of the transfer operator which we adapt from Pollicott \& Urba\'{n}ski and a Tauberian theorem \ref{Garahm-Vaaler} due to Graham \& Vaaler. Further, we investigate an asymptotic of the length for which the counting function is related to. Given $T>0$ the maximum length contributing to the counting function is itself subject to an asymptotic formula, see proposition \ref{asymp length}. \\
We use a similar approach to that of Pollicott \& Urba\'{n}ski but we include much more details. We mention proof of important facts used by Pollicott \& Urba\'{n}ski in section \ref{3}. Not only, do we bring the proofs but also we adapt it for our own setting lifting conformality. As well, we reprove some of the inequalities in section \ref{5} concerning counting finite words just for the purpose of giving a clearer and shorter proof. Furthermore, we include many preliminaries that some experts can skip over. One major reason we decided to write this detailed paper is mostly for the purpose of having a self-contained article. The other reason is for the paper following this paper which targets the counting problem in the context of random dynamical systems. This way we can simply refer to any desired (deterministic) preliminaries here.\\
About the structure  of the article, we start with some preliminaries from dynamical systems over symbolic space like pressure, Gibbs, and equilibrium state. This further includes topics like real or complex-valued summable functions, H\"{o}lder functions. Moreover, we define properties like strong regularity and D-generic property which are the main assumptions in Pollicott \& Urba\'{n}ski's formula. We introduce the transfer operator and talk a bit about the perturbation theory of analytic operators. This requires the concept of essential spectrum. Later, we bring notions of graph directed Markov systems and conformal graph directed Markov systems. Finally, we briefly mention two Tauberian theorems before we finish the section with some examples.\\
The next section is devoted to applying the perturbation theory to the transfer operator to obtain a spectral representation of the transfer operator over its maximal eigenvalues. We show first eigenvalues are simple beforehand though. This requires introducing a weighted operator involving the transfer operator.\\
The fourth section is talking about the relation between a complex function (Poincar\'{e} series), and some counting functions. The idea is by taking the Riemmann-Stieltjes integral against our target counting function, we acquire a Poincar\'{e} series. Further, we bring some estimates to find some upper and lower bounds for counting periodic words in terms of our ordinary counting function.\\
In the fifth section, we use the spectral representation from the third section to argue how Graham \& Vaaler's Tauberian theorem is applicable to imply the main theorem. We use this theorem to obtain Pollicott \& Urba\'{n}ski's formula as a corollary. \\
The sixth section investigates an asymptotic formula for the length of the words that contribute to the counting function. We provide two asymptotic formulas for length and we propose a conjecture about the growth of the counting function relative to another counting function with some specified length. We finish with three examples to see how our estimates of bounds are sharp in the last section.

\section{Preliminaries}
We would like to mention that throughout this paper we try to stay loyal to the following conventions:
\begin{itemize}
    \item $\omega, \tau, \gamma$ : finite words
    \item $\rho, \rho'$ : infinite words
    \item $T$ : positive real
    \item $s=x+iy$ : complex number
    \item $\sigma$ : shift map
    \item $m,\mu$ : measures
    \item $f,g,h$ : real or complex functions on $E_A^{\mathbb{N}}$
    \item $\mathbbm{1}_B$ : indicator function of set $B$
    \item $C^{0,\alpha}$ : space of H\"{o}lder functions of exponent $\alpha$
    \item $K, Q,c,c_1,C_1,c_{\delta}$ : constants
    \item $\mathcal{L}, \mathcal{P}, \mathcal{Q}, \mathcal{D}, \mathcal{F}, \mathcal{E}$ :  operators
    \item $\lambda$ : eigenvalue
    \item $\Gamma^+$ : some right half-plane
    \item $\eta_{\rho}(B,s)$ : complex function in $s$
    \item $N_{\rho}(B,T)$ counting function in $T$
\end{itemize}
Let $E$ be a countable (finite or infinite) set calling each of its elements a symbol, a letter or an alphabet. By $E^{\mathbb{N}}$ we mean the set of all infinite sequences of the form 
$$e_1e_2e_3...e_n...$$
where each $e_i$ belongs to $E$. We usually represent the first $n$ symbols of such a sequence, also called (finite) \textbf{word} or block, by $\omega$ throughout this work, i.e. 
$$\omega=e_1e_2...e_n$$
where we sometimes tend to identify $\omega_i$ with $e_i$ and just have 
$$\omega=\omega_1\omega_2...\omega_n.$$
When we write $|\omega|=n$ we just mean the word $\omega$ has $n$ letters. By $E^n$ we represent all the words of length $n$ and by $E^*$ we represent $\cup_{n=1}^{\infty} E^n$. As well we use the notation $|.\wedge .|$ to represent the number of common initial symbols in two sequences, i.e. for $\rho=e_1e_2...$ and $\rho'=e'_1e'_2...$ we have
$$|\rho \wedge \rho '|=m \Leftrightarrow e_1=e'_1,\; e_2=e'_2, \; ..., \; e_m=e'_m, \; e_{m+1}\neq e'_{m+1}.$$
One can as well introduce a metric by
$$d(\rho, \rho')=e^{- |\rho \wedge \rho '|}.$$
Further, we set
$$d_{\alpha}=d^{\alpha}, \; \; 0<\alpha<1,$$
i.e. we have
$$d_{\alpha}(\rho, \rho')=e^{-\alpha |\rho \wedge \rho '|}.$$
Therefore we equip $E^{\mathbb{N}}$ with a metric space, which is called \textbf{symbolic space}. Note that the topology on $E^{\mathbb{N}}$ induced by this metric is the same as the Tychonoff topology where each $E$ is equipped with ordinary discrete topology. This means for any $\alpha$ and $\beta$ the topologies of $d_{\alpha}$ and $d_{\beta}$ are the same, however, the metrics are not equivalent for different $\alpha$ and $\beta$.\\
One can then see that the shift map $\sigma:E^{\mathbb{N}} \rightarrow E^{\mathbb{N}}$ given by
$$\sigma(e_1e_2...)=e_2e_3...$$
is a continuous map. \\
Further, we want to restrict ourselves to sequences where certain words are not appearing. We first introduce a map $A: E\times E \rightarrow \{0,1\}$ (sometimes called incidence or transition matrix). We use $A_{ee'}$ notation instead of $A(e,e')$.\\
A \textbf{subshift of finite type} consists of the sequences $e_1e_2e_3...$ in $E^{\mathbb{N}}$ such that
$$A_{e_1e_2}=1, \; A_{e_2e_3}=1, \; ..., \; A_{e_ne_{n+1}}=1, \; ...$$
Of course, if $A$ only assumes value $1$, represented by $A=1$, then this is just the space introduced earlier, that is why we sometimes call $(E^{\mathbb{N}} , \sigma)$ full shift space.\\
Further, when $A_{e_1e_2}=1$ we say $e_1e_2$ is \textbf{A-admissible} or just \textbf{admissible}. As well, by $E^*_A$ we mean all admissible finite words of all lengths, by $E^*_{\rho}$ we mean all $\omega \in E^*_A$ such that $\omega \rho$ is an admissible sequence, by $E^n_{\rho}$ we mean all $\omega \in E^n_A$ such that $\omega \rho$ is an admissible sequence, by $E^*_{\text{per}}$ we mean all $\omega \in E^*_A$ such that $\omega_n \omega_1$ is admissible and we say $\omega$ is \textbf{periodic word}, by $\bar{\omega}$ we mean the sequence $\omega \omega\omega ...$ and by $E^n_A$ we mean all admissible words of length $n$. Finally, for each finite word $\omega$ of length $n$ we define the \textbf{cylinder} 
$$[\omega]:=\{\rho \in E_A^{\mathbb{N}} : \; \rho_1...\rho_n=\omega  \}.$$
\begin{proposition}\label{polish}
For the subshift of finite type $E_A^{\mathbb{N}}$ the followings hold:
\begin{enumerate}
    \item[a.] All the cylinders form a countable clopen basis.
    \item[b.] Every open set can be written as a countable union of mutually disjoint cylinders.
    \item[c.] It is a Polish space.
\end{enumerate}
\end{proposition}
\begin{proof}
(a). It is clear that for each positive integer $n$, we have countably many finite words of length $n$, therefore there are only countably many cylinders. Next, we show each cylinder is a neighborhood in $E_A^{\mathbb{N}}$. Let $\omega$ be a finite word of length $n$, choose any fixed $\rho \in [\omega]$, we show $[\omega]=N(\rho,e^{-\alpha(n-1)})$. Note that $\rho'$ is in $[\omega]$ iff $d_{\alpha}(\rho,\rho')<e^{-\alpha(n-1)}$ iff $|\rho\wedge\rho'|>n-1$ iff $|\rho\wedge\rho'|\geq n$ iff $\rho' \in [\omega]$. To see $[\omega]$ is closed, consider a sequence $\{\rho_{(i)}\}_i$ in $[\omega]$ converging to $\rho$. This means $|\rho_{(i)}\wedge\rho| \to \infty$ which clearly implies $\rho \in [\omega]$. Now for every open set $V$ and every $\rho \in V$, note that there is $\epsilon>0$ such that $\rho \in N(\rho,\epsilon)\subseteq V$. We choose $n$ large enough such that $e^{-\alpha(n-1)}< \epsilon$, then obviously $[\rho_1\rho_2...\rho_n]=N(\rho,e^{-\alpha(n-1)})\subset N(\rho,\epsilon)\subset V$.\\
(b). The fact that an open $V$ can be written as a countable union of cylinders is clear from part a. Then part b follows from the fact that for any two cylinders $[\omega]$ and $[\tau]$ that meet each other, we have either $[\omega] \subset [\tau]$ or $[\tau] \subset [\omega]$. To show this, assume $\rho$ belongs to both of the cylinders $[\omega]$ and $[\tau]$. Further, assume $|\omega|\leq |\tau|$. Since $\rho \in [\tau]$, we should have $\rho=\tau\rho'$ for some $\rho' \in E_A^{\mathbb{N}}$, similarly $\rho \in [\omega]$ implies that $\rho=\omega\rho''$ for some $\rho'' E_A^{\mathbb{N}}$. Thus $\tau\rho'=\rho=\omega\rho''$ and since $|\omega| \leq |\tau|$ so $\tau=\omega\omega'$ for some finite word $\omega'$. This implies $[\tau]\subseteq [\omega]$.\\
(c). Note that a countable product of separable spaces is separable and a countable product of complete metrizable spaces is complete metrizable.
\end{proof}
We would like to mention that we only work with probability measures over Borel sets all through this work.
\begin{definition}
For a measurable transformation $T: X\to X$ on a measure space $(X, \mathcal{B})$ we say a measure $\mu$ is \textbf{$T-$invariant} if for every $A \in \mathcal{B}$:
$$\mu(T^{-1}(A))=\mu(A).$$
Further we say $\mu$ is \textbf{ergodic} if $\mu$ is $T-$invariant measure such that if $T^{-1}(A)=A$ then either $\mu(A)=0$ or $\mu(A)=1$.
\end{definition}
\begin{definition}\label{finitely irreducible}
We call a subshift \textbf{finitely irreducible} if there exists a finite set $\Omega$ containing words such that for all $e,e' \in E$ there is $\omega \in \Omega$ such that $e\omega e'$ is admissible. As well subshift is called \textbf{finitely primitive} if it is finitely irreducible and all words in $\Omega$ are of fixed length.
\end{definition}
Throughout this paper, we restrict ourselves to work with finitely irreducible subshifts.
\begin{remark}
Note that this notion is just a generalization of an irreducible matrix when $E$ is countable. In fact, a finitely irreducible condition guarantees that the shift map is topologically mixing, and a finitely primitive guarantees that the shift map is topologically exact. Additionally, it is clear that if the shift space is finitely irreducible then the backward orbit of every element is dense, i.e. 
$$\overline{\cup_{n=0}^{\infty}\sigma^{-n}(\rho)}=E_A^{\mathbb{N}}.$$
\end{remark}
\begin{proposition}\label{spectral radius of A}
If $E$ is finite,
$$\log r(A)=\lim_n \frac{1}{n}\log \#E^n_A,$$
where $r(A)$ is spectral radius of matrix $A$.
\end{proposition}
\begin{proof}
We refer to theorem 3.2.22 \cite{Munday}.
\end{proof} 
Next, we want to talk about the H\"{o}lder continuous maps. In Analysis textbooks \citep[p. 52]{Holder} we have different notions of H\"{o}lder continuity of exponent $\alpha$ for real or complex-valued functions on a Euclidean space $D$:
\begin{itemize}
\item H\"{o}lder at a point $x_0$:
$\sup_{x\in U} \{|f(x)-f(x_0)|/|x-x_0|^{\alpha}\}$ is finite, where $U$ is a neighborhood of $x_0$ in $D$.
\item H\"{o}lder:
$\sup_{x,y \in D} \{|f(x)-f(y)|/|x-y|^{\alpha}\}$ is finite.
\item Locally H\"{o}lder:
$\sup_{x,y \in K} \{|f(x)-f(y)|/|x-y|^{\alpha}\}$ is finite for every compact $K \subseteq D$.

\end{itemize}

We call each of the above suprema the \textbf{H\"{o}lder coefficient}. Of course, $D$ can be replaced with the metric space $E_A ^{\mathbb{N}}$ to obtain similar notions on the shift space. We denote the set of complex-valued H\"{o}lder continuous functions of \textbf{H\"{o}lder exponent} $\alpha$ on $E_A ^{\mathbb{N}}$ by $C^{0,\alpha}(E_A^{\mathbb{N}},\mathbb{C})$ or simply $C^{0,\alpha}$.
We remind that the usual \textbf{H\"{o}lder coefficient} is defined by:
$$|g|_{\alpha}=\sup_{\rho,\rho' \in E_A^{\mathbb{N}}} \left\{\frac{|g(\rho)-g(\rho')|}{d_{\alpha}(\rho,\rho')} \right\}.$$
We would like to define another \textbf{H\"{o}lder coefficient} that is justified later. We set:
$$V_{\alpha,n}(f):=\sup\{|f(\rho_1)-f(\rho_2)|e^{\alpha(n-1)} : |\rho_1\wedge\rho_2|\geq n\geq 1 \},$$
and
$$V_{\alpha}(f):=\sup_{n\geq 1} V_{\alpha,n}(f).$$
There is another notion of H\"{o}lder continuity useful for our purposes. 
\begin{definition}\label{HO type}
A complex-valued function f on $E_A^{\mathbb{N}}$ is called \textbf{H\"{o}lder-type} continuous with exponent $\alpha>0$ if $V_{\alpha}(f) < \infty$.
\end{definition}
We define a \textbf{norm} on $C^{0,\alpha}(E_A^{\mathbb{N}},\mathbb{C})$ by
\begin{equation}\label{norm1}
\|g\|_{\alpha}:=\|g\|_{\infty}+V_{\alpha}(g)
\end{equation}
We are ready to find relations between these different notions of H\"{o}lder continuity.
\begin{proposition}\label{norm equivalence}
The followings hold:
\begin{itemize}
    \item[(a)] On $E_A ^{\mathbb{N}}$ every complex-valued function is H\"{o}lder continuous iff it is H\"{o}lder-type continuous and bounded.
    \item[(b)] The norm given above in \ref{norm1} is equivalent to usual $\|.\|_{C^{0,\alpha}}=\|.\|_{\infty}+|.|_{\alpha}$ norm over $C^{0,\alpha}(E_A^{\mathbb{N}},\mathbb{C})$.
    \item[(c)] $\left(C^{0,\alpha}(E_A^{\mathbb{N}},\mathbb{C}), \|.\|_{\alpha}\right)$ is Banach space.
    \item[(d)] A H\"{o}lder-type continuous function is locally H\"{o}lder continuous and H\"{o}lder continuous at every point.
\end{itemize}
\end{proposition}
\begin{proof}
a) Assume $f$ is H\"{o}lder continuous function, then there is $M$ such that
$$|f(\rho_1)-f(\rho_2)| \leq M d(\rho_1,\rho_2)^{\alpha}=M e^{-\alpha |\rho_1 \wedge \rho_2|},$$
for every $\rho_1$ and $\rho_2$. Therefore 
$$|f(\rho_1)|\leq |f(\rho_1)-f(\rho_2)|+|f(\rho_2)|\leq M+|f(\rho_2)|.$$
This gives boundedness of $f$. For H\"{o}lder-type, assuming $|\rho_1 \wedge \rho_2| \geq n$, it follows 
$$|f(\rho_1)-f(\rho_2)|e^{\alpha(n-1)} \leq M e^{-\alpha},$$
i.e. $V_{\alpha}(f)\leq M e^{-\alpha}$.\\
For the converse, assuming that $|f|\leq K$ for some constant $K$, and $|\rho_1 \wedge \rho_2|= n \geq 1$ we have
$$|f(\rho_1)-f(\rho_2)|e^{\alpha(n-1)} \leq V_{\alpha}(f).$$
Therefore 
$$|f(\rho_1)-f(\rho_2)| \leq V_{\alpha}(f) e^{-\alpha(n-1)}=V_{\alpha}(f)e^{\alpha}d(\rho_1,\rho_2)^{\alpha}.$$
In case $|\rho_1 \wedge \rho_2|=0$, we use boundedness of $f$ to get
$$|f(\rho_1)-f(\rho_2)| \leq 2K=2Kd(\rho_1,\rho_2)^{\alpha}.$$
Thus 
$$|f(\rho_1)-f(\rho_2)| \leq \max\{2K, V_{\alpha}(f)e^{\alpha}\} d(\rho_1,\rho_2)^{\alpha},$$
for every $\rho_1$ and $\rho_2$.\\
b) From the proof above we realize that $V_{\alpha}(f)\leq |f|_{\alpha}e^{-\alpha}$ which leaves
$$\|f\|_{\alpha} \leq \|f\|_{\infty}+|f|_{\alpha} e^{-\alpha} \leq \|f\|_{\infty}+|f|_{\alpha}.$$
Furthermore $|f|_{\alpha}\leq \max\{2K, V_{\alpha}(f)e^{\alpha}\}$ gives us
$$\|f\|_{\infty}+|f|_{\alpha} \leq 3\|f\|_{\infty}+V_{\alpha}(f)e^{\alpha}\leq (3+e^{\alpha})\|f\|_{\alpha}.$$
c) This is a well-known fact, see for example \cite[~ p. 73]{Holder} for a Euclidean space.\\
d) This is easy to show.
\end{proof}
\begin{remark}
We want to justify why we used the terminology H\"{o}lder-type:
\begin{itemize}
    \item The H\"{o}lder-type continuous functions subject of study in this paper in the case of infinite alphabets are summable. This makes them unbounded and so they are not H\"{o}lder.
    \item  Let $E=\mathbb{N}$. One can see that $f: E^{\mathbb{N}} \to \mathbb{R}$ defined by $f(kn_2n_3n_4...)=\ln 1/n_k^2$, is H\"{o}lder continuous at each point (consider $[kn_2n_3n_4...n_k]$) and locally H\"{o}lder continuous but is not H\"{o}lder-type continuous.
    \item  Note that locally H\"{o}lder continuous on $E_A ^{\mathbb{N}}$ wouldn't imply continuity necessarily, however H\"{o}lder continuity at a point clearly implies continuity.
    \item Regarding H\"{o}lder continuity at a point even if we were able to find a uniform bound for H\"{o}lder coefficients that worked for all the points it still doesn't imply H\"{o}lder-type continuity necessarily.
    \item  Over shift space with finite alphabets H\"{o}lder continuity and H\"{o}lder-type continuity coincide.
\end{itemize}
\end{remark}
Below we need to use a sequence of finite words in the lemma. For that we use the notation $\omega_{(i)}$, to denote that it is not the $i^{th}$ coordinate of $\omega$ which we represent by $\omega_i$.
\begin{lemma}
Let $\{\omega_{(i)}\}_{i \in I}$ be any collection of finite words with bounded length, i.e. there exists a positive integer $k$ such that $|\omega_{(i)}|\leq k$ for each $i$. If the cylinders $\{[\omega_{(i)}]\}_{i \in I}$ are mutually disjoint, then the indicator function of $H:=\cup_{i \in I}[\omega_{(i)}]$ is H\"{o}lder continuous, i.e. $\mathbb{1}_{H} \in C^{0,\alpha}(E_A^{\mathbb{N}},\mathbb{C})$.
\begin{proof}
We want to show there exists $M>0$ such that 
$$|\mathbb{1}_{H}(\rho)-\mathbb{1}_{H}(\rho')|\leq M d(\rho,\rho'),$$
for every $\rho,\rho' \in E_A^{\mathbb{N}}$. If $\rho,\rho' \in H$, there is nothing to prove as the left-hand side is $0$. Similarly if $\rho,\rho' \notin H$. If $\rho \in H$ and $\rho' \notin H$, then there is $i$ such that $\rho \in [\omega_{(i)}]$. But $|\rho\wedge\rho'|<|\omega_{(i)}|$, otherwise $\rho' \in [\omega_{(i)}]$. Therefore 
$$e^{-k}\leq e^{-|\omega_{(i)}|} \leq e^{-|\rho\wedge\rho'|}=d(\rho,\rho')  .$$
Thus if we just pick $M=e^k$, then for each $\rho,\rho'$ we have
$$|\mathbb{1}_{H}(\rho)-\mathbb{1}_{H}(\rho')|\leq M d(\rho,\rho').$$
\end{proof}
\end{lemma}
\begin{lemma}\label{bounded S_n}
If $f: E_A^{\mathbb{N}} \rightarrow \mathbb{C}$ is H\"{o}lder-type continuous with $V_{\alpha}(f)<\infty$ then there exists $K_{f}>0$ such that for any $\omega \in E^n_A$ and any $\rho,\rho' \in E^{\mathbb{N}}_A$ where $\omega\rho, \omega\rho'$ are admissible we have
$$|S_nf(\omega\rho)-S_nf(\omega\rho')|\leq K_{f} d(\rho,\rho ').$$
\end{lemma}
\begin{proof}
We refer to \citep[p. ~26]{MU}.
\end{proof}
A sequence $\{a_n\}$ of real numbers is called \textbf{subadditive} if for every positive integer $m,n$:
$$a_{m+n}\leq a_m+a_n.$$
\begin{lemma}[Fekete's Lemma]\label{fekete}
For every subadditive sequence $\{a_n\}$, the limit of the sequence $\{\frac{a_n}{n}\}$ exists and it is equal to $\inf_n \left\{\frac{a_n}{n}\right\}$.
\end{lemma}
\begin{proof}
We refer to \citep[~ p. 5]{MU}.
\end{proof}
\begin{lemma}\label{liminf}
Let $f_i(T)$ be a collection of non-negative functions defined on $T>0$. Then 
$$\sum_i \liminf_{T \rightarrow \infty}f_i(T) \leq \liminf_{T \to \infty} \sum_i f_i(T) $$
\end{lemma}
\begin{proof}
Of course, if the collection is finite, this is clear. We show it for an infinite countable collection. As each $f_i$ is non-negative so for each $n$
$$\sum_{i=1}^n f_i(T) \leq \sum_i f_i(T).$$
Taking liminf from both sides
$$\sum_{i=1}^n \liminf_{T \to \infty} f_i(T) \leq \liminf_{T \to \infty}\sum_i f_i(T).$$
This holds for each $n$, therefore we get the inequality.
\end{proof}
Unfortunately, analogous inequality for limsup doesn't hold even if $\sum_i f_i(T)$ is uniformly bounded above. Alternatively, we mention the following inequality.
\begin{lemma}\label{f,g}
For any two non-negative functions $f(T), g(T)$ defined on $T>0$, we have 
$$\liminf_{T \to \infty} \left(f(T)+g(T)\right)\leq \liminf_{T \to \infty} f(T) +\limsup_{T \to \infty} g(T)\leq \limsup_{T \to \infty} \left(f(T)+g(T)\right).$$
\begin{proof}
Let $\underline{l}=\liminf_{T \to \infty} \left(f(T)+g(T)\right)$, and $\overline{g}=\limsup_{T \to \infty} g(T)$. For $\epsilon>0$ there is $T_0$ such that for $T>T_0$ we have 
$$\underline{l}-\epsilon \leq f(T)+g(T) \leq f(T)+\overline{g}+\epsilon,$$
$$\underline{l}-\overline{g}-2\epsilon \leq f(T),$$
which establishes the left inequality. A similar argument gives the right inequality.
\end{proof}
\end{lemma}
A real-valued function $f$ on $E_A^{\mathbb{N}}$ is called \textbf{summable} if 
$$\sum_{e \in E} \exp(\sup_{[e]} f)<\infty.$$
One purpose of this definition is to define an operator on the space of bounded complex-valued continuous functions on $E_A^{\mathbb{N}}$. Therefore we can extend this definition to complex-valued functions.
\begin{definition}\label{summable 2}
A complex-valued function $f$ on $E_A^{\mathbb{N}}$ is called \textbf{summable} if 
$$\sum_{e \in E} \exp(\sup_{[e]} \operatorname{Re} (f))<\infty.$$
\end{definition}

\begin{definition}\label{RPF}
For a complex-valued H\"{o}lder-type summable function $f$ we introduce \textbf{Ruelle-Perron-Frobenius} operator, also known as \textbf{transfer} operator
$$\mathcal{L}_{f}:C_b(E_ A^{\mathbb{N}},\mathbb{C}) \rightarrow C_b(E_ A^{\mathbb{N}},\mathbb{C})$$ $$\mathcal{L}_{f} (g)(\rho)=\sum\limits_{\substack{ e \in E_{\rho}}}\exp \left(f(e\rho)\right)g(e\rho), $$ 
where the sum is taken over all  $e \in E$ that  $e\rho$ is admissible, i.e. $A_{e \rho_1}=1$.
\end{definition}
\begin{remark}
Here we would like to mention:
\begin{itemize}
    \item  If this $f$ over shift with infinite letters is summable, then definition \ref{summable 2} yields that $\text{Re}(f)$ should go to $-\infty$, i.e. $f$ is unbounded.  Therefore it is not H\"{o}lder continuous, see proposition \ref{norm equivalence}.
    \item  As well it is clear that when $E$ is finite then every real-valued $f$ is summable.
    \item  Further one can see that this operator preserves $C^{0,\alpha}(E_A^{\mathbb{N}},\mathbb{C})$.
\end{itemize}
\end{remark}
Next, we want to consider the adjoint operator $\mathcal{L}^*_{f}$ acting on $C_b(E_A^{\mathbb{N}},\mathbb{C})^{*}$ which is the space of all regular bounded additive set functions \citep[p. 262]{Dunford} (by an additive set function we mean a complex-valued function $g$ defined on the algebra, not necessarily $\sigma$-algebra, generated by the closed sets such that g is finitely additive, not necessarily countably additive). Below we mention a result which for case $E$ finite is due to Ruelle \cite{Ruelle} and for $E$ infinite is due to Mauldin-Urba\'{n}ski \citep[~ p. 50]{MU}.
\begin{theorem}\label{R. D. Mauldin and M. Urbanski}
If $f: E_A^{\mathbb{N}} \to \mathbb{R}$ is real-valued summable and H\"{o}lder-type continuous function, then the adjoint operator $\mathcal{L}^*_{f}$ admits an eigenmeasure $m$ with eigenvalue $\exp(P(f))$.
\end{theorem}
This $P(f)$ is introduced below in definition \ref{pressure}.
\begin{definition}\label{def Gibbs}
A \textbf{Gibbs state} for a real-valued function $f$ on $E_A^{\mathbb{N}}$ is a probability measure $m$ on $E_A^{\mathbb{N}}$ for which there is $Q>1$ and $P\in \mathbb{R}$ such that:
$$Q^{-1}\leq \frac{m([\omega])}{\exp\left(S_nf(\omega\rho)-Pn\right)}\leq Q, \; \; \;  \forall \omega \in E_A ^{n}, \; \;  \forall \omega\rho  \;\text{admissible} .$$
\end{definition}
It is clear that a Gibbs state has full support, i.e.
$$\text{supp}(m)=E_A^{\mathbb{N}}.$$
Another important fact is that once we get an eigenmeasure from theorem \ref{R. D. Mauldin and M. Urbanski} it follows that it is actually a Gibbs state for $f$ \citep[p. 28]{MU}. Using this Gibbs state an invariant ergodic Gibbs state $\mu$ for $f$ can be constructed as well \citep[p. 14]{MU}. Furthermore, it is clear that if $f$ is H\"{o}lder-type so is any constant multiple of $f$. However, the summable property of $f$ doesn't necessarily carry on to any constant multiple of $f$. We set 
$$\Gamma:=\{ x \in \mathbb{R}: xf \; \; \text{summable} \}.$$
Clearly, if $E$ is finite then $\Gamma=\mathbb{R}$ and if $E$ is infinite then definition \ref{summable 2} tells us $x_1 \in \Gamma$ implies $x_2 \in \Gamma$ for any $x_2>x_1$, i.e. $\Gamma$ is half line. Therefore using the above explanation we obtain Gibbs state for $xf$ ($x \in \Gamma$) as well:
\begin{equation}\label{Gibbs}
Q_x^{-1}\leq \frac{m_x([\omega])}{\exp\left(xS_nf(\omega\rho)-P(x)n\right)}\leq Q_x, \; \; \;  \forall \omega \in E_A ^{n}, \; \;  \forall \omega\rho  \;\text{admissible} .
\end{equation}

\begin{definition}\label{pressure}
The \textbf{topological pressure} of a real-valued function $f$ on $E_A ^{\mathbb{N}}$ is defined by
$$P(f)=\lim_{n\to \infty}\frac{1}{n}\ln \big( \sum\limits_{\omega \in E_A^n}\exp( \sup_{[\omega]}S_nf) \big).$$
\end{definition}
This limit exists by Fekete's lemma \ref{fekete}. 
\begin{definition}
A invariant ergodic measure $\mu$ is called \textbf{equilibrium state} for a real-valued function $f$ on $E_A^{\mathbb{N}}$ if it is a Gibbs state for $f$ and it established the following equation:
$$P(f)= h_{\mu}(\sigma)+\int f\text{d}\mu,$$
where $h_{\mu}$ is Kolmogorov entropy of the shift map $\sigma$.
Note that in general under a much weaker assumption for $f$ we have the following equation known as \textbf{variational principle}:
$$P(f)=\sup \{ h_{\mu}(\sigma)+\int f\text{d}\mu \},$$
where the supremum is taken over invariant ergodic measures $\mu$. Furthermore, we set
$$\chi_{\mu}=-\int f\text{d}\mu,$$
and call it \textbf{Lyapunov exponent}.
\end{definition}
One can see that $P$ in definition \ref{def Gibbs} is actually the same as the topological pressure of  $f$ \cite[~ p. 13]{MU}. This means
$$P(x)=P(xf), \; \; x\in \Gamma.$$
We can actually show this function is strictly decreasing on $\Gamma$ assuming some weak condition. This is a well-known fact for function systems, but here we don't assume $f$ is induced by a function system and so we prove it. First, we need the following lemma.
\begin{lemma}\label{measure goes to zero}
If $\mu$ is an invariant ergodic Gibbs measure then 
$$\lim_n \sup_{\omega \in E_A^n} \mu{([\omega])}=0.$$
\end{lemma}
\begin{proof}
Let $b_n=\sup_{\omega \in E_A^n} \mu{([\omega])}$. Note that this supremum is attained so $b_n$ is decreasing, therefore $b_n$ is convergent to some $b$. Fix $0<\epsilon<b$ and for each $n$ define
$$F_n:=\{ \omega \in  E_A^n : \epsilon \leq \mu([\omega])\}.$$
Clearly, $F_n$ is finite. If $\omega e \in F_{n+1}$ then $\epsilon \leq \mu([\omega e]) \leq \mu([\omega])$ which implies $\omega \in F_n$, i.e. each $F_{n+1}$ extends some of $F_n$. If this extension process stops at moment $m$ or in other words, $F_m=\emptyset$ then $\mu([\omega]) < \epsilon$ for all $\omega \in E_A^m$, i.e. $b\leq b_m \leq \epsilon$. Therefore this process cannot stop and so we get at least one element $\rho=e_1e_2e_3... \in E_A^{\mathbb{N}}$ such that $\epsilon \leq \mu([e_1...e_n])$ for each $n$. This means $\epsilon \leq \mu(\{ \rho\})$. We will show $\rho$ is periodic and the periodic orbit of $\rho$ $O_+(\rho)$ has full measure which is a contradiction.\\ 
Let $A:=\cup_{n \geq 0}\sigma^{-n}(\{\rho\}) $. Clearly either $\sigma^{-1}(A)=A$ or $\sigma^{-1}(A)\cup\{\rho\}=A$. In the latter case $\mu(\sigma^{-1}(A))+\mu(\{ \rho \})=\mu(A)$ which yields $\mu(\{ \rho \})=0$ using invariant property of $\mu$. In the former case, $\rho$ must be periodic with some period $m$. Since $\sigma^{-1}(A)=A$, ergodicity either yields $\mu(\{ \rho \}) \leq \mu(A)=0$, or otherwise  $\mu(A)=1$. Note that for each $i>0$ we have $\sigma^{i-1} (\rho) \in \sigma^{-1}\left(\{\sigma^i(\rho)\}\right)$ so
$$\mu(\{ \sigma^{i-1}(\rho) \})\leq \mu\left(\sigma^{-1}(\{\sigma^i(\rho)\})\right)=\mu(\{ \sigma^i(\rho) \}),$$ 
and since $\rho=\overline{e_1e_2...e_m}$ we have
$$\mu(\{ \rho \})\leq \mu(\{ \sigma(\rho) \}) \leq ... \leq \mu(\{ \sigma^{m-1}(\rho) \}) \leq \mu(\{ \rho \}).$$
Therefore the inequalities in the above line are all equality. For each $n\geq 0$ we know $\sigma^{-n}(\{ \rho \})$ meets $O_+(\rho)$ in exactly one point and since $\mu\left( \sigma^{-n}(\{ \rho \})\right)=\mu(\{ \rho\})$ thus the whole mass of $\sigma^{-n}(\{ \rho \})$ is on $\sigma^{-n}(\{ \rho \}) \cap O_+(\rho)$. Therefore
$$1=\mu(A)=\mu(O_+(\rho)).$$
\end{proof}
\begin{proposition}\label{pressure decreasing}
If $P(x_0)\leq 0$ for some $x_0$ then $P(x)$ is strictly decreasing on $\Gamma$.
\end{proposition}
\begin{proof}
We start with the following estimate and we use \ref{Gibbs} for it:
$$\exp\left(x_0\sup_{[\omega]}S_nf-nP(x_0)\right)\leq Q m([\omega]).$$
Next we use the above lemma to find $N$ such that for every $n\geq N$ and every $\omega \in E_A^n$:
$$\exp\left( x_0\sup_{[\omega]}S_nf-nP(x_0)\right) \leq Q \mu([\omega])\leq Q\sup_{\omega \in E_A^n} \mu([\omega])\leq e^{-1}.$$
Then for all $k>0$ and $\omega ' \in E_A^{kN}$:
$$\exp\left( x_0\sup_{[\omega']}S_{kN}f-kNP(x_0)\right) \leq \exp\left(x_0k\sup_{[\omega]} S_Nf-Nk P(x_0) \right) \leq  e^{-k}.$$
Consider $x_1<x_2$ in $\Gamma$, we use the above estimate to find
$$\sum\limits_{\substack{\omega' \in E_A^{kN}}}\exp(x_2 \sup_{[\omega']}S_{Nk}f) = \sum\limits_{\substack{\omega' \in E_A^{kN}}}\exp(x_1 \sup_{[\omega']}S_{Nk}f) \exp\left((x_2-x_1) \sup_{[\omega']}S_{Nk}f\right)$$
$$=\sum\limits_{\substack{\omega' \in E_A^{kN}}}\exp(x_1 \sup_{[\omega']}S_{Nk}f) \exp\left(\frac{x_2-x_1}{x_0}\big( x_0\sup_{[\omega']}S_{Nk}f-kNP(x_0)\big)\right)\exp\left( \frac{x_2-x_1}{x_0}kNP(x_0) \right)$$
$$ \leq \sum\limits_{\substack{\omega' \in E_A^{kN} }}\exp(x_1 \sup_{[\omega']}S_{Nk}f) \exp \left(-k\frac{x_2-x_1}{x_0}\right)=\exp \left(-k\frac{x_2-x_1}{x_0}\right) \sum\limits_{\substack{\omega' \in E_A^{kN} }}\exp(x_1 \sup_{[\omega']}S_{Nk}f)$$
Now if we take $\log$, divide by $kN$ and let $k \to \infty$, we obtain
$$P(x_2)\leq -\frac{x_2-x_1}{Nx_0}+P(x_1)<P(x_1).$$
\end{proof}
\begin{definition}\label{strongly regular}
A real-valued function $f: E^{\mathbb{N}}_A \rightarrow \mathbb{R}$ is called \textbf{regular} if $P(x)=0$ for some $x>0$ and is called \textbf{ strongly regular} if it is regular and $0<P(x)<\infty$ for some $x>0$.
\end{definition}
\begin{remark}
It is worth mentioning
\begin{itemize}
\item
If $P(x)$ is strictly decreasing, it can have only one root say $\delta$. Further, strong regularity means
$$\inf \Gamma < \delta.$$
\item
The above proposition can be proved under weaker assumption: $\inf_{x \in \Gamma} P(x)\leq 0.$
\end{itemize}
\end{remark}
\begin{proposition}
If $f$ is strongly regular, the first derivative of $P$ at $\delta$ is $P'(\delta)=-\chi_{\mu_{\delta}}$.
\end{proposition}
\begin{proof}
We refer to proposition 2.6.13 in \citep[~ p. 47]{MU}.
\end{proof}
Next, we want to consider a spectral theory of the transfer operator. We start with considering  a family of functions $\{ sf \}$ where $s$ is usually a complex number in the right half plane $\Gamma^+=\Gamma \times \mathbb{R}$. Then definitions \ref{summable 2} and \ref{RPF} are applicable for such functions, however, definitions \ref{pressure} is not applicable as $sf$ is not real anymore unless for real $s$.\\
It is clear that when $s \in \Gamma^+$ then the series
$$\sum_{e \in E}\sup_{[e]}|\exp(sf)|=\sum_{e \in E}\exp(\operatorname{Re}(s)\sup_{[e]}f)$$
converges. Thus having a H\"{o}lder-type summable function, spectral theory of transfer operator on the right half plane $\Gamma$ makes sense. Note that 
\begin{itemize}
    \item $\mathcal{L}_{s}:=\mathcal{L}_{sf}$ is an operator on $C^{0,\alpha}(E_A^{\mathbb{N}},\mathbb{C})$ for any  $s \in \Gamma ^+$.
    \item  The pressure function $P$ is defined on $\Gamma$.
\end{itemize}
Another important property of the transfer operator to be discussed is the D-generic property. This property prohibits the possibility of admitting specific eigenvalue. We adopt its definition from \cite[~ p. 32]{UP}. Before mentioning the definition, we need an equivalency.

\begin{proposition}
The following conditions are equivalent:
\begin{itemize}
    \item[(i)] The multiplicative group $G_a(y):=<\exp(-|\omega|a)\exp(yS_{|\omega|}f(\bar{\omega})) : \omega \in E^*_{\text{per}}>$ is generated by an integer power of $e^{2\pi}$.
    \item[(ii)] $\exp(P(x)+ia)$ is an eigenvalue of $\mathcal{L}_{x+iy}: C^{0,\alpha}(E_A^{\mathbb{N}},\mathbb{C}) \rightarrow C^{0,\alpha}(E_A^{\mathbb{N}},\mathbb{C})$, for some $x \in \Gamma$.
    \item[(iii)] There exists $u \in C_b(E_A^{\mathbb{N}},\mathbb{C})$ such that $yf-a+u\circ \sigma -u \in C_b(E_A^{\mathbb{N}},2\pi \mathbb{Z})$.
\end{itemize}
\end{proposition}
\begin{proof}
We refer to Proposition 2 in \cite[~ p. 138]{P-RPF} and Proposition 2.3.5 in \citep[~ p. 32]{UP}.
\end{proof}
\begin{definition}\label{d-generic}
We say a potential $f$ is \textbf{D-generic} if one of the statements of the above proposition fails for all non-zero $y$ and $a=0$. In other words, $\mathcal{L}_{x+iy}: C^{0,\alpha}(E_A^{\mathbb{N}},\mathbb{C}) \rightarrow C^{0,\alpha}(E_A^{\mathbb{N}},\mathbb{C})$ doesn't admit $\exp(P(x))$ as eigenvalue if $y\neq 0$.\\
Further, we say the potential $f$ is \textbf{strongly D-generic}, if either of the above statements (i) or (ii) from the above proposition fails for all non-zero $y$ and all real $a$. In other words, $\mathcal{L}_{x+iy}: C^{0,\alpha}(E_A^{\mathbb{N}},\mathbb{C}) \rightarrow C^{0,\alpha}(E_A^{\mathbb{N}},\mathbb{C})$ doesn't admit any eigenvalue of magnitude $\exp(P(x))$ for any $y\neq 0$. 
\end{definition}
One can obtain an alternative statement for D-generic and strongly D-generic properties.
\begin{proposition}\label{prop. d-generic}
A potential $f$ is D-generic iff the additive subgroup generated by the following set is not cyclic.
$$\{S_{|\omega|}f(\bar{\omega}): \; \omega \in E^*_{\text{per}} \}.$$
And it is strongly D-generic iff the additive subgroup generated by the following set is not cyclic for any real $\beta$.
$$\{S_{|\omega|}f(\bar{\omega})-n\beta: \; \omega \in E^*_{\text{per}} \}.$$
\end{proposition}
\begin{proof}
    We only prove the first claim. The other one can be proved in a similar manner. Assuming $f$ is not D-generic, then for some non-zero $y$ and $a=0$, statement $(i)$ holds. Then for any $\omega \in E_{\text{per}}^*$, there exists integer $k_{\omega}$ such that $yS_{|\omega|}f(\bar{\omega})=2\pi k_{\omega}$. Therefore $<S_{|\omega|}f(\bar{\omega}): \; \omega \in E^*_{\text{per}} > \subseteq <2\pi k/y : k \in \mathbb{Z}>$. Since the latter group is cyclic, so is the first one. For the converse, we assume that $<S_{|\omega|}f(\bar{\omega}): \; \omega \in E^*_{\text{per}} >$ is cyclic. Then there exists a non-zero $y$ such that $S_{|\omega|}f(\bar{\omega})=y k_{\omega}$ for any $\omega \in E_{\text{per}}^*$, where $k_{\omega} \in \mathbb{Z}$. Therefore, $G_0(2\pi / y)$ is generated by $e^{2\pi}$.
\end{proof}
Next, we would like to bring some facts from spectral theory. We mostly refer to \cite{Dunford}, \cite{TK}, \cite{Browder} or \cite{Baumg}. Assume $\mathfrak{B}$ is a Banach space, $\mathcal{L}$ a bounded operator on $\mathfrak{B}$.\\
The \textbf{spectrum} of the bounded operator $\mathcal{L}$, denoted by $\text{Sp}(\mathcal{L})$, is defined to be all the complex numbers $\zeta$ such that the operator $(\mathcal{L}-\zeta I)$ is not bijective. Further the \textbf{spectral radius} of $\mathcal{L}$ is defined to be
$$r(\mathcal{L}):=\sup \{ |\zeta| \; : \; \zeta \in \text{Sp}(\mathcal{L}) \}$$
There is an alternative expression of spectral radius known as Gelfand's formula.
$$r(\mathcal{L})=\lim_n \|\mathcal{L}^n\|^{\frac{1}{n}}.$$
Next, we mention the essential spectrum definition. We indicate that there are several other definitions of this concept in the math community, however, the radius of the essential spectrum (defined below) remains the same for all the definitions. We adapt the following definition from \citep[~ p. 107]{Browder}.
\begin{definition}\label{essential spectrum def}
The complex number $\zeta$ belongs to the \textbf{essential spectrum} of the operator $\mathcal{L}$, denoted by $\text{Sp}_{\text{ess}}(\mathcal{L})$, if at least one of the following condition holds:
\begin{itemize}
    \item[(i)] the operator $(\mathcal{L}-\zeta I)$ has a range which is not closed in $\mathfrak{B}$.
    \item[(ii)] $\cup_{i\geq 0} \text{ker}(\mathcal{L}-\zeta I)^i$ is infinite dimensional.
    \item[(iii)] the point $\zeta$ is a limit point of the spectrum of $\mathcal{L}$.
\end{itemize}
Furthermore, the \textbf{essential spectral radius} is
$$r_{\text{ess}}(\mathcal{L}):=\sup \{ |\zeta| \; : \; \zeta \in \text{Sp}_{\text{ess}}(\mathcal{L}) \}$$
\end{definition}
Nussbaum showed the essential spectral radius follows a Gelfand's type formula. Before bringing his formula we need to introduce a semi-norm. Consider, $\mathfrak{K}$ the ideal of all bounded compact operators on $\mathfrak{B}$, then 
$$\| \mathcal{L} \|_{\mathfrak{K}}:= \inf_{\mathcal{C} \in \mathfrak{K} } \|\mathcal{L}+\mathcal{C} \|,$$
defines a semi-norm on the space of bounded linear operators on $\mathfrak{B}$ \citep[~ p. 474]{Nussbaum}.
\begin{proposition}\label{ess spec rad}
$r_{\text{ess}}(\mathcal{L})=\lim_{n \to \infty} \| \mathcal{L}^n \|_{\mathfrak{K}} ^{1/n}$
\end{proposition}
\begin{proof}
We refer to \citep[~ p. 477]{Nussbaum}.
\end{proof}
Next, we briefly talk about the perturbation theory of linear operators. Our main sources are \cite{TK}, \cite{Dunford} and \cite{Baumg}.\\
It is now clear from definition \ref{essential spectrum def} that for every $r$ where $r_{\text{ess}}(\mathcal{L})<r\leq r(\mathcal{L})$ we should have only finitely many $\zeta \in \text{Sp}(\mathcal{L})$ with $|\zeta|\geq r$, each of which is
isolated eigenvalue with finite algebraic multiplicity. Kato calls these finite $\zeta$'s, \textbf{finite system of eigenvalues} \citep[~ p. 181]{TK} or \citep[~ p. 363]{Baumg}. This concept shows up in \citep[~ p. 572]{Dunford} as \textbf{spectral set}. According to Schwartz-Dunford, \textbf{spectral set} is any clopen subset of the spectrum. The purpose is to obtain a perturbation theorem for a holomorphic family of operators $\mathcal{L}_s$ in complex variable $s$. The original idea of the perturbation theory of self-adjoint operators over Hilbert space goes back to Schrödinger. The first major math result in this area was obtained by Rellich. Later on Sz. Nagy and Kato independently worked on this topic to generalize Rellich's result to a general closed operator over Banach space \cite{Kato2}. Many of these results can be found in \citep[~ VII.6]{Dunford} or \citep[~ Ch. VII]{TK} or \citep[~ Ch. 10]{Baumg}. We first want to define a holomorphic family of operators. Note that there are several definitions for this but all in the context of bounded operator-valued over a fixed Banach space coincide \citep[~ 10.1]{Baumg}, \citep[~ 10.3]{Baumg}.

\begin{definition}
Let $(X, \|. \|)$ be Banach space, $\mathfrak{B}(X)$ be the space of all bounded linear operator on $X$, $G$ a region in the complex plane and $s\mapsto \mathcal{L}_s$ a function from $G$ into $\mathfrak{B}(X)$. We say $\mathcal{L}_s$ is \textbf{holomorphic} in $G$ if there exists an operator-valued function $s\mapsto \mathcal{L}'_s$ such that
$$\|\frac{\mathcal{L}_{s+h}-\mathcal{L}_s}{h}-\mathcal{L}'_s \| \rightarrow 0,$$
for all $s \in G$ and $h\rightarrow 0$.
\end{definition}
We are ready to express one major result in the perturbation theory of a holomorphic family of bounded operators.
\begin{theorem}\label{kato theorem}
Let $\mathcal{L}_s$ be a holomorphic family of bounded operators from a region $G$ into $\mathfrak{B}(X)$. Let $s_0 \in G$ and $\lambda_0,...\lambda_n$ be finite system of eigenvalues of $\mathcal{L}_{s_0}$, each of which with algebraic multiplicity $1$. Then there is small enough neighborhood of $s_0$ such that $\mathcal{L}_s$ has the spectral representation
$$\mathcal{L}_s=\sum_{i=1}^n \lambda_i(s)\mathcal{P}_{i,s}+\mathcal{D}_s, $$ 
where each $\lambda_i(s)$ is holomorphic function, $\mathcal{P}_i(s)$ is holomorphic operator-valued function and a projection, $\mathcal{D}(s)$ holomorphic operator-valued function and further 
$$\lambda_i(s_0)=\lambda_i,$$
for each $i=1,...,n.$
\end{theorem}
In general, if the multiplicity of an eigenvalue is higher than $1$ the eigenvalues may have algebraic singularities at $s_0$. The idea of the proof is first reducing it to the case where $X$ is finite-dimensional and then one can apply perturbation theory of holomorphic operators in finite dimension. For a detailed proof, first see theorem 1 in \citep[~ p. 367]{Baumg}, then theorem 1 in \citep[~ p. 243]{Baumg}, \citep[~ p. 129]{Baumg} and \citep[~ p. 131]{Baumg}. Another source of proof for the general form of the result is theorem 9 in \citep[~ p. 587]{Dunford}. As well theorem 1.8 in \citep[~ p. 370]{TK} provides a proof.\\
By projection in the above theorem, we mean an operator with property
$$\mathcal{P}_i^2=\mathcal{P}_i.$$
Next, we would like to see how the above spectral representation of operators is related to the spectral decomposition of operators. The following proposition is a consequence of the celebrated spectral mapping theorem \citep[~ p. 209]{Conway}.
\begin{proposition}\label{spectral mapping theorem}
Suppose $\mathfrak{B}(X)$ is a Banach algebra of operators on the Banach space $X$. Let $\mathcal{L} \in \mathfrak{B}(X)$. Further, assume the spectrum of $\mathcal{L}$ can be written as 
$$\text{Sp}(\mathcal{L})=F_1\cup F_2,$$
for disjoint nonempty closed sets $F_1,F_2$. Then there is a nontrivial idempotent $\mathcal{E} \in \mathfrak{B}(X)$ such that
\begin{itemize}
    \item if $\mathcal{B}\mathcal{L}=\mathcal{L}\mathcal{B}$, then $\mathcal{B}\mathcal{E}=\mathcal{E}\mathcal{B}$.
    \item if $\mathcal{L}_1=\mathcal{L}\mathcal{E}$ and $\mathcal{L}_2=\mathcal{L}(1-\mathcal{E})$, then $\mathcal{L}=\mathcal{L}_1+\mathcal{L}_2$ and $\mathcal{L}_1\mathcal{L}_2=\mathcal{L}_2\mathcal{L}_1=0$.
    \item $\text{Sp}(\mathcal{L}_1)=F_1\cup \{0\}$, $\text{Sp}(\mathcal{L}_2)=F_2\cup\{0\}$.
\end{itemize}
\end{proposition}
\begin{definition}
We first consider directed multi-graph $(V,E,i,t)$ and an incidence matrix $A: E\times E \rightarrow \{0,1\}$, where $V$ is the finite set of vertices, $E$ is the countable (finite or infinite) set of directed edges and $i,t$ (initial and tail) are functions 
$$i,t: E\rightarrow V,$$
such that 
$$A_{ab}=1\; \; \Rightarrow \; \; t(a)=i(b)$$
In addition, we have a finite family of Euclidean compact metric spaces $\{X_v\}_{v \in V}$ and countable family of contractions $\{\phi_e\}_{e \in E}$ and $\kappa \in (0,1)$ such that
$$|\phi_e(x)-\phi_e(y)| \leq \kappa |x-y|,$$
for all $e \in E$ and $x,y \in X_{t(e)}$. Then 
$$\mathcal{S}=\{\phi_e: X_{t(e)} \rightarrow X_{i(e)} \}_{e \in E}$$
is called attracting \textbf{graph directed Markov system}.
\end{definition}
We extend the functions $i,t: E \rightarrow V$ in a natural way to $E_A^*$ as follows:
$$t(\omega):=t(\omega_n), \; \; \; i(\omega):=i(\omega_1).$$
If $\omega \in E_A^n$ we define:
$$\phi_{\omega}=\phi_{\omega_1}\circ ...\circ \phi_{\omega_n}:X_{t(\omega)} \rightarrow X_{i(\omega)}.$$
Now for any $\rho \in E_A^{\mathbb{N}}$ the sets $\{\phi_{\rho_1\rho_2...\rho_n}(X_{t(\rho_n)})\}_{n\geq 1}$ form a descending sequence of non-empty compact sets and therefore $\cap_{n\geq 1} \phi_{\rho_1\rho_2...\rho_n}(X_{t(\rho_n)})$ is non-empty. Further since 
$$\mathrm{diam}(\phi_{\rho_1\rho_2...\rho_n}(X_{t(\rho_n)})) \leq \kappa^n \mathrm{diam}(X_{t(\rho_n)}) \leq \kappa^n max\{\mathrm{diam}(X_v)\}_{v \in V},$$
we find that this intersection is actually a singleton and we denote it by $\pi(\rho)$, in this way we have defined a map 
$$\pi: E_A^{\mathbb{N}} \rightarrow \sqcup_{v \in V}	X_v,$$
where $\sqcup_{v \in V}	X_v$ is the disjoint union of the compact spaces $\{X_v\}_v$.
\begin{definition}
The set 
$$J=\pi(E_A^{\mathbb{N}})$$
is called the \textbf{limit set} of system $\mathcal{S}$.
\end{definition}
\begin{definition}\label{conformal condition}
We call a graph directed Markov system \textbf{conformal} if the following conditions are satisfied for some $d \in \mathbb{N}$:
\begin{enumerate}
    \item[(a)] For every $v \in V$, $X_v$ is compact connected subset of $\mathbb{R}^d$ and $X_v=\overline{\text{Int}(X_v)}$.
    \item[(b)] (Open Set Condition) For all different $e, e' \in E$,
    $$\phi_e(\text{Int}(X_{t(e)})\cap \phi_{e'}(\text{Int}(X_{t(e')})=\emptyset.$$
    \item[(c)] (Conformality) For every $v \in V$ there is an open connected $W_v$ containing $X_v$. Further for each $e \in E$, $\phi_e$ extends to a $C^1$ conformal diffeomorphism from $W_{t(e)}$ into $W_{i(e)}$ with Lipschitz constant bounded by $\kappa$.
    \item[(d)] (Bounded Distortion Property) There are two constants $L\geq 1$ and $\alpha> 0$ such that for every $e \in E$ and every $x, y \in X_{t(e)}$
    $$\left|\frac{|\phi_e'(y)|}{|\phi_e'(x)|}-1\right| \leq L\|y-x\|^{\alpha},$$
    where $|\phi_e'(x)|$ denotes the scaling factor of the derivative of $\phi_e'$ at $x$.
\end{enumerate}
\end{definition}
To a conformal graph directed Markov system we assign a real-valued function by
$$f: E_A^{\mathbb{N}} \rightarrow \mathbb{R}, \; \; \; f(\rho)=\log |\phi_{\rho_1}'(\pi(\sigma\rho))|.$$
Also known as \textbf{potential}.\\
Before finishing the section, we mention two main Tauberian theorems needed later on. First Ikehara \& Wiener's theorem \citep[p. ~127]{wiener} and then Graham \& Vaaler's theorem \citep[p. ~294]{Vaaler} which is just a refinement of the Ikehara-Wiener theorem. The motivation for the Ikehara-Wiener theorem was to provide a simpler proof of the prime number theorem. We know that PNT was proved in the late 19$^{th}$ century. However, Ikehara \& Wiener used a theorem of Wiener to obtain the following result in the early 1930s that implies PNT \citep[~ p. 127]{Korevaar}.
\begin{theorem}[Ikehara-Wiener]\label{ikehara-wiener delta}
Let $\alpha(T)$ be a monotone increasing function continuous from right such that 
$$\eta(s)=\int _0 ^{\infty} e^{-sT}d\alpha(T)$$
converges for $\operatorname{Re}(s) > \delta>0$. If 
$$\eta(s)-\frac{A}{s-\delta}=g(s)$$
has continuous extension to $\operatorname{Re}(s) =\delta$, then 
$$e^{-\delta T}\alpha(T) \rightarrow \frac{A}{\delta}, \; \; \; \text{as} \; \; T \rightarrow \infty.$$
\end{theorem}
In the 1980s, Graham \& Vaaler on their way to study extremal (minorant and majorant) functions in Fourier analysis for some special classes of functions, obtained a refinement of the Ikehara-Wiener theorem as a corollary. One may want to know that the early work in the construction of extremal functions was done by Beurling and later on by Selberg (unpublished). For the proof of the following result see \citep[p. ~294]{Vaaler}.
\begin{theorem}[Graham-Vaaler]\label{Garahm-Vaaler}
Let $\alpha$ be a Borel measure on $[0,\infty)$ and that the Laplase-Stieltjes transform 
$$\eta(s)=\int_{0^-}^{\infty} e^{-sT}\text{d}\alpha(T), \; \; \; \; \; \; \; s=x+2\pi i y,$$
exists for $\operatorname{Re}(s)>\delta$. Suppose that for some number $y_0>0$, there is a constant $A>0$ such that the analytic function $\eta(s)-A/(s-\delta)$ extends to a continuous function on the set $\{\delta+2\pi iy: |y| < y_0 \}$. Then 
\begin{align*}
Ay_0^{-1}\{\exp(\delta y_0^{-1})-1\}^{-1}   & \leq \liminf_{T \to \infty} e^{-\delta T}\alpha[0,T] \\
 & \leq \limsup_{T \to \infty}  e^{-\delta T}\alpha[0,T] \\
 & \leq Ay_0^{-1}\{\exp(\delta y_0^{-1})-1\}^{-1}\exp(\delta y_0^{-1})
\end{align*}
\end{theorem}
\begin{remark}
It is worth noting that
\begin{itemize}
\item If $\eta(s)-A/(s-\delta)$ has continuous extension to the whole line $x=\delta$ then we may let $y_0 \to \infty$, this implies Ikehara-Wiener theorem \ref{ikehara-wiener delta}.
\item Graham \& Vaaler or Korevaar \citep[p. 30]{Korevaar} assumed that $A$ should be positive or non-negative. But since $\eta$ is real non-negative on the real line, this assumption can be relaxed, i.e. $A$ can be any complex number. Then one can see it has to be real non-negative. Furthermore, it is clear that for us the  measure $\alpha$ (possibly infinite measure) is just taken to be the Borel measure generated by the right continuous, increasing function $N_{\rho}(B,T)$, see \ref{eq1} and \citep[Thm 1.16]{Folland}. Moreover, Graham \& Vaaler provide an example to show their bounds are both sharp.
\end{itemize}
\end{remark} 
\begin{example}\label{example alpha different}
Consider the iterated function system where in the multi-graph $(V, E, i, t)$, $V$ is Singleton $\{v\}$, $E$ is finite $i=t$ are maps from $E$ to the only element of $V$ and the mapping $A: E\times E \rightarrow \{0,1\}$ is just constant $1$. This is an iterated function system. Then for the conformal graph directed Markov systems, we consider $X_v=[0,1]$ and 
$$\phi_e(t)=\alpha_et+\beta_e,$$
where $\alpha_e, \beta_e$ are chosen appropriately enough from $(0,1)$ so that we have all conditions for conformal graph directed Markov system satisfied, see definition \ref{conformal condition}. Then we know from below the definition \ref{conformal condition} the potential is
$$f(\rho)=\log |(\phi_{\rho_1})'\left(\pi(\sigma \rho)\right)|=\log \alpha_{\rho_1}$$
and 
$$S_nf(\rho)=\sum_{e \in E} n_e \log \alpha_e,$$
where $n_e$ is just number of letter $e$ appearing in the word $\rho_1...\rho_n$.  We can find the pressure:
\begin{equation}\label{eqr 6}
P(x)=\lim_{n}\frac{1}{n}\log \sum_{|\omega|=n}\|\phi_{\omega}'\|^x=\log (\sum_{e \in E}\alpha_e^x)
\end{equation}
As well we know the following should hold
\begin{equation}\label{eqr 10}
m_x([e\omega_1...\omega_n])=
\exp (-P(x)) \int_{[\omega_1...\omega_n]} |(\phi_{e})'\left(\pi(\tau)\right)|^x dm_x(\tau),
\end{equation}
for $e \in E$, Gibbs state $m_x$ and pressure $P(x)$. This actually leaves 
$$\dfrac{m_x([e\omega_1...\omega_n])}{m_x([\omega_1...\omega_n])}=\dfrac{\alpha_e^x}{\sum_{e \in E}\alpha_e^x}$$
\end{example}
\begin{example}\label{alpha_i constant}
Consider an iterated function system containing conformal maps 
$$\phi_e(t)=\alpha t+\beta_e , \; \; \;   \alpha, \beta_e \in (0,1),$$
where $\beta_e$ are appropriate enough for conformal conditions, see Definition \ref{conformal condition}, $E=\{0,1,...,k-1\}$, with some irreducible incidence matrix $A$. Then we know for this system we should have the potential 
$$f(\omega)=\log |(\phi_{\omega_1})'(\pi(\sigma \omega)|=\log \alpha$$ 
and a Gibbs state has the form
\begin{equation}
m_x([e\omega_1...\omega_n])=
\exp (-P(x)) \int_{[\omega_1...\omega_n]} |(\phi_{e})'\left(\pi(\tau)\right)|^x dm_x(\tau)
\end{equation}
for appropriate $e \in E$, Gibbs state $m_x$ and pressure $P(x)$. This actually yields 
$$P(x)=\log \left( m_x([\omega_1...\omega_n])/m_x([e\omega_1...\omega_n]) \right)+x\log \alpha$$
Note that the first term of the above sum does not depend on $t$ and it is actually equal to $\lim_n \log \#E_A^n/n=\log r(A)$, see Proposition \ref{spectral radius of A}, where r(A) is the spectral radius of the incidence matrix A, we can see this simply by the pressure formula:
\begin{equation}\label{Pp}
P(x)=\lim_{n}\frac{1}{n}\log \sum_{\omega \in E_A^n}\|\phi_{\omega}'\|^x=\lim_{n}\frac{ \log \#E_A^n}{n}+x \log \alpha =\log r(A)+x\log \alpha.
\end{equation}
\end{example}

\section{Spectral Analysis of Transfer Operator}\label{3}

In this section we assume we have a summable strongly regular H\"{o}lder-type function (potential) $f: E_A^{\mathbb{N}} \to \mathbb{R}$ with $P(f)=P(1)=0$. We recall $C^{0,\alpha}(E_A^{\mathbb{N}},\mathbb{C})$ is Banach space of H\"{o}lder continuous complex-valued functions over $E_A^{\mathbb{N}}$, and $\mathfrak{B}:=\mathfrak{B}\left(C^{0,\alpha}(E_A^{\mathbb{N}},\mathbb{C})\right)$ is Banach space of all bounded linear operators over $C^{0,\alpha}(E_A^{\mathbb{N}},\mathbb{C})$. For every $s \in \Gamma^+$ it was stated in the previous section that $\mathcal{L}_s$ belongs to $\mathfrak{B}$. One major step is to establish holomorphy of the operator $\mathcal{L}_s$.
\begin{lemma}
For every $n\in \mathbb{N}$, the operator-valued function $s \mapsto \mathcal{L}_s ^n$ is holomorphic on $\Gamma ^+$.
\end{lemma}
\begin{proof}
For each $\omega \in E_A ^{n}$ one can consider the (idempotent) function $i_{[\omega]}$ in $C^{0,\alpha}(E_A^{\mathbb{N}},\mathbb{C})$ where it is defined to be $1$ on $\xi$ such that $\omega \xi$ is admissible and $0$ otherwise. Then for each $s$ in the right half  plane $\Gamma^+$ and $g \in C^{0,\alpha}(E_A^{\mathbb{N}},\mathbb{C})$ we define $\mathcal{F}_{\omega,s}g$:
\begin{equation}\label{F_w}
\mathcal{F}_{\omega,s}g(\rho):=i_{[\omega]}(\rho) \exp(sS_nf(\omega\rho))g(\omega\rho).
\end{equation}
We want to show $\mathcal{F}_{\omega,s}$ is an operator on $C^{0,\alpha}(E_A^{\mathbb{N}},\mathbb{C})$. First note that
$$ \|\mathcal{F}_{\omega,s}g(\rho)\|_{\infty} \leq \exp(\operatorname{Re}(s)\sup_{[\omega]} S_nf)\|g\|_{\infty}.$$
To find H\"{o}lder coefficient of $\mathcal{F}_{\omega,s}g$ we let $|\rho \wedge \rho'|\geq k \geq 1$: 
$$ |\mathcal{F}_{\omega,s}g(\rho)-\mathcal{F}_{\omega,s}g(\rho')|\leq \\ | \exp(sS_nf(\omega \rho))g(\omega \rho)-\exp(sS_nf(\omega \rho'))g(\omega \rho')|$$
$$=|\left( \exp(sS_nf(\omega \rho))-\exp(sS_nf(\omega \rho'))\right)g(\omega \rho)+\exp(sS_nf(\omega \rho'))\left(g(\omega \rho)-g(\omega \rho') \right)|$$
$$\leq \exp(\operatorname{Re}(s)\sup_{[\omega]} S_nf).|s|.|S_nf(\omega\rho)-S_nf(\omega\rho')|.\|g\|_{\infty}$$
$$+\exp(\operatorname{Re}(s)\sup_{[\omega]} S_nf)|g(\omega \rho)-g(\omega \rho') | .$$
By Lemma \ref{bounded S_n}, we get
$$|\mathcal{F}_{\omega,s}g(\rho)-\mathcal{F}_{\omega,s}g(\rho')|\exp(\alpha k) \leq |\mathcal{F}_{\omega,s}g(\rho)-\mathcal{F}_{\omega,s}g(\rho')|\exp(\alpha |\rho \wedge \rho'|)$$
$$\leq \exp(\operatorname{Re}(s)\sup_{[\omega]} S_nf).|s|.K.\|g\|_{\infty}+ \exp(\operatorname{Re}(s)\sup_{[\omega]} S_nf)V_{\alpha}(g)$$ 
$$ \leq \exp(\operatorname{Re}(s)\sup_{[\omega]} S_nf)\|g\|_{\alpha}(1+|s|K),$$
where $K$ depends only on $f$.
Therefore we can write:
$$\|\mathcal{F}_{\omega,s}g\|_{\alpha}=\|\mathcal{F}_{\omega,s}g\|_{\infty} +V_{\alpha}(\mathcal{F}_{\omega,s}g)$$
$$ \leq \exp(\operatorname{Re}(s)\sup_{[\omega]} S_nf)\|g\|_{\infty}+\exp(\operatorname{Re}(s)\sup_{[\omega]} S_nf)\|g\|_{\alpha}(1+|s|K)$$ 
$$ \leq \exp(\operatorname{Re}(s)\sup_{[\omega]} S_nf)\|g\|_{\alpha}(2+|s|K),$$
so
\begin{equation}\label{norm}
\| \mathcal{F}_{\omega,s}\|_{\alpha}\leq \exp(\operatorname{Re}(s)\sup_{[\omega]} S_nf)(2+|s|K).
\end{equation}
Next, we want to show the map $s \mapsto \mathcal{F}_{\omega,s}$  is holomorphic on $\Gamma ^+$. As expected derivative is 
$$\mathcal{F}'_{\omega,s}g(\rho)=i_{[\omega]}(\rho). \exp(sS_nf(\omega\rho)).S_nf(\omega\rho).g(\omega\rho),$$
we first need to show this defines an operator on $C^{0,\alpha}(E_A^{\mathbb{N}},\mathbb{C})$ and then to check it is actually bounded. Note that $|S_nf|$ is bounded on $[\omega]$ by some C, see definition \ref{def Gibbs}. If we review all the inequalities above and replace all the $g(\omega...)$ with $S_n(\psi)(\omega...)g(\omega...)$  we get:
$$\|\mathcal{F}'_{\omega,s}g\|_{\alpha}=\|\mathcal{F}'_{\omega,s}g\|_{\infty} +V_{\alpha}(\mathcal{F}'_{\omega,s}g)$$
$$ \leq \exp(\operatorname{Re}(s)\sup_{[\omega]} S_nf).\|g\|_{\infty}.C+\exp(\operatorname{Re}(s)\sup_{[\omega]} S_nf).\|g\|_{\alpha}.(C+|s|KC+K)$$ 
$$ \leq \exp(\operatorname{Re}(s)\sup_{[\omega]} S_nf)\|g\|_{\alpha}(2C+|s|KC+K).$$
Fix $s_0$ in $\Gamma ^+$, we write:
$$ \left(\mathcal{F}_{\omega,s}-\mathcal{F}_{\omega,s_0}-(s-s_0)\mathcal{F}'_{\omega,s_0}\right)g(\rho)$$
$$= i_{[\omega]}(\rho). \big(\exp(sS_nf(\omega\rho))-\exp(s_0S_nf(\omega\rho))$$
$$-(s-s_0)\exp(s_0S_nf(\omega\rho))S_nf(\omega\rho)\big).g(\omega\rho),$$
therefore 
$$s \mapsto \mathcal{F}_{\omega,s} \in \mathfrak{B}$$ 
is holomorphic iff 
$$s \mapsto i_{[\omega]}(...) \exp(sS_nf(\omega...)) \in C^{0,\alpha}(E_A^{\mathbb{N}},\mathbb{C})$$ 
is holomorphic. But 
$$i_{[\omega]}(...) \exp(sS_nf(\omega...))=i_{[\omega]}(...) \exp(si_{[\omega]}(...)S_nf(\omega...))$$
and $i_{[\omega]}(...)S_nf(\omega...) \in C^{0,\alpha}(E_A^{\mathbb{N}},\mathbb{C})$, thus since $i_{[\omega]}$ is a constant function of $s$ problem boils down to holomorphy of the function $s \mapsto \exp(s\mathcal{T})$ for $\mathcal{T} \in C^{0,\alpha}(E_A^{\mathbb{N}},\mathbb{C})$, and this is clearly holomorphic.\\
Thus the map $s \mapsto \mathcal{F}_{\omega,s}$ defines a holomorphic $\mathfrak{B}$-valued function on the right half plane $\Gamma^+$. Now because for $s \in \Gamma^{+}$, $\operatorname{Re}(s)f$ admits Gibbs state, $\mathfrak{B}$-valued function 
$$s \mapsto \mathcal{L}_s ^n =\displaystyle \sum_{\omega \in E_A^n}\mathcal{F}_{\omega,s}$$
converges and so is holomorphic on $\Gamma ^+$.
\end{proof}
\begin{proposition}\label{spectral radius L}
The spectral radius of $\mathcal{L}_s$ is at most $e^{P(x)}$ and $r_{\text{ess}}(\mathcal{L}_s)<e^{P(x)}$.
\end{proposition}
\begin{proof}
For the case $E$ is finite we just refer to \cite[~ p. 140]{P-RPF}. Assuming $E$ is infinite, the former part is a straightforward consequence of Doeblin inequality (also known as Ionescu Tulcea-Marinescu inequality or Lasota-Yorke inequality) shown in \citep[p. 32]{MU}:
\begin{equation}\label{lasota1}
\|\mathcal{L}_s ^n g  \|_{\alpha} \leq e^{nP(x)} (Qe^{-\alpha n}\|g\|_{\alpha}+C \|g\|_{\infty}) ,
\end{equation}
which leaves 
\begin{equation}\label{lasota}
\|\mathcal{L}_s ^n\|_{\alpha} \leq e^{nP(x)} (Q+C).
\end{equation}
First for every $\omega \in E_A^*$ choose $\hat{\omega} \in [\omega]$ arbitrarily. Then for every $n\geq 1$ consider the operator $\mathcal{E}_n$ on $C^{0,\alpha}(E_A^{\mathbb{N}},\mathbb{C})$ defined by:
$$\mathcal{E}_n(g) \coloneqq \sum_{\omega \in E_A^n} g(\hat{\omega})\mathbb{1}_{[\omega]}.$$
Therefore $\mathcal{E}_ng$ is constant on each cylinder $[\omega]$. It is clear that $\|\mathcal{E}_ng\|_{\infty} \leq \|g\|_{\infty}$. We want to show $V_{\alpha}(\mathcal{E}_ng) \leq V_{\alpha}(g)$. Remembering definition \ref{HO type} if $m\geq n$ then clearly $V_{\alpha, m}=0$, in case $1\leq m<n$ and $|\rho_1 \wedge \rho_2|\geq m$ there should be $\omega_1, \omega_2 \in E_A^n$ such that $\rho_1 \in [\omega_1]$ and $\rho_2 \in [\omega_2]$, therefore $|\hat{\omega}_1\wedge\hat{\omega}_2|\geq m$ and
$$|\mathcal{E}_ng(\rho_1)-\mathcal{E}_ng(\rho_2)|e^{\alpha(m-1)}=|g(\hat{\omega}_1)-g(\hat{\omega}_2)|e^{\alpha(m-1)} \leq V_{\alpha,m}(g).$$
Thus we have:
\begin{equation}\label{basic}
\|\mathcal{E}_ng\|_{\alpha} \leq \|g\|_{\alpha}.   
\end{equation}
Next without loss of generality assume $E=\mathbb{N}$ and for each $N\geq 1$ define
$$E_A^n(N) \coloneqq\{ \omega \in E_A^n : \omega_1,\omega_2,...,\omega_n \leq N\}$$
$$E_A^n(N+) \coloneqq  E_A^n \setminus E_A^n(N)$$
$$\mathcal{E}_{n,N}g \coloneqq \sum _{\omega \in E_A^n(N)} g(\hat{\omega})\mathbb{1}_{[\omega]}.$$
Note that $n$ and $N$ are independent. Moreover notice that this time since we have finite sum the operator $\mathcal{E}_{n,N}$ on $C^{0,\alpha}(E_A^{\mathbb{N}},\mathbb{C})$ is of finite rank and so compact. We use triangle inequality to write:
$$\|\mathcal{L}_s ^n-\mathcal{L}_s ^n \mathcal{E}_{n,N} \|_{\alpha} \leq \|(\mathcal{L}_s ^n -\mathcal{L}_s ^n \mathcal{E}_n)+(\mathcal{L}_s ^n \mathcal{E}_n -\mathcal{L}_s ^n \mathcal{E}_{n,N} )\|_{\alpha}$$
\begin{equation}\label{triangle L}
\leq \|\mathcal{L}_s ^n(\mathcal{I} -\mathcal{E}_n)\|_{\alpha}+\|\mathcal{L}_s ^n( \mathcal{E}_n - \mathcal{E}_{n,N} )\|_{\alpha} ,
\end{equation}
where $\mathcal{I}$ is just the identity operator. Note that \ref{basic} implies $\|g-\mathcal{E}_ng\|_{\alpha} \leq 2\|g\|_{\alpha}$. Furthermore, for any $\rho \in E_A^{\mathbb{N}}$ if set $\omega=\rho_1...\rho_n$ then we have $|\rho\wedge \hat{\omega}|\geq n$ and
$$|g(\rho)-\mathcal{E}_ng(\rho)|e^{\alpha(n-1)}=|g(\rho)-g(\hat{\omega})|e^{\alpha(n-1)}\leq V_{\alpha,n}(g)\leq V_{\alpha}(g).$$
Since $\rho$ is arbitrary, we obtain 
$$\|g-\mathcal{E}_ng\|_{\infty} \leq V_{\alpha}(g)e^{\alpha}e^{-\alpha n}\leq \|g\|_{\alpha}e^{\alpha} e^{-\alpha n}.$$ 
Thus using two recent inequalities and $\ref{lasota1}$ we find
$$\|\mathcal{L}_s ^n(I -\mathcal{E}_n)g\|_{\alpha} \leq e^{nP(x)} (Qe^{-\alpha n}2\|g\|_{\alpha}+C \|g\|_{\alpha}e^{\alpha}e^{-\alpha n})$$
\begin{equation}\label{approximating L(I-E)}
\leq C_1e^{nP(x)}\|g\|_{\alpha}e^{-\alpha n},
\end{equation}
for some constant $C_1>0$. Recalling $\mathcal{F}_{\omega,s}$ from the proof of previous lemma, we can write
$$\mathcal{F}_{\omega ',s}(\mathcal{E}_n g-\mathcal{E}_{n,N}g)=\sum_{|\omega|=n}g(\hat{\omega})\mathcal{F}_{\omega ',s}(\mathbb{1}_{[\omega]})-\sum_{\omega \in E_A^n(N)}g(\hat{\omega })\mathcal{F}_{\omega ',s}(\mathbb{1}_{[\omega ]})$$ 
$$=\sum_{\omega \in E_A^n(N+)}g(\hat{\omega })\mathcal{F}_{\omega ',s}(\mathbb{1}_{[\omega ]})=g(\hat{\omega ' })\mathcal{F}_{\omega ',s}(\mathbb{1}_{[\omega ']})\; \; \text{or}  \; \; 0, $$
depending on $\omega ' \in E_A^n(N+)$ or not, so
$$\mathcal{L}_s ^n(\mathcal{E}_n g-\mathcal{E}_{n,N}g)=\sum_{\omega ' \in E_A^n}F_{\omega ',s}(\mathcal{E}_n g-\mathcal{E}_{n,N}g)=\sum_{\omega  \in E_A^n(N+)}g(\hat{\omega  })\mathcal{F}_{\omega,s }(\mathbb{1}_{[\omega]}).$$
Then $\ref{norm}$ leaves:
$$\|\mathcal{L}_s ^n(\mathcal{E}_n g-\mathcal{E}_{n,N}g)\|_{\alpha} \leq \| g\|_{\infty} \sum_{\omega  \in E_A^n(N+)}\| \mathcal{F}_{\omega,s }(\mathbb{1}_{[\omega]})\|_{\alpha}$$
$$ \leq \| g\|_{\infty}(2+|s|K) \sum_{\omega  \in E_A^n(N+)}\exp (x \sup _{[\omega]}S_nf) .$$
Now since $A$ is finitely irreducible, there exists a finite set $\Omega \subseteq E_A^*=\cup_n E_A^n$ such that for every $e \in E$ and $\rho \in E_A^{\mathbb{N}}$, there is $\omega \in \Omega$ with $e\omega\rho$ being admissible. Thus there exists a finite set $F \subseteq E_A^{\mathbb{N}}$ such that for every $e \in E$, there is $\tau \in F$ with $e \tau$ being admissible. For every $\omega \in E_A^*$ choose $\tau_{\omega} \in F$ with $\omega \tau_{\omega}$ admissible. Therefore using \ref{Gibbs} we can continue
$$\leq \| g\|_{\infty}(2+|s|K)Q^2 \sum_{\omega  \in E_A^n(N+)}\exp (x S_nf(\omega \tau_{\omega})).$$
Moreover, if we consider 
$$c_N \coloneqq \sup_{j \geq N} \exp (\sup f[j]),$$
then the fact that $f$ is summable implies that $c_N \rightarrow 0$. Now for each $\omega  \in E_A^n(N+)$ there is $\omega_i>N$ so 
$$\exp ( S_nf(\omega \tau_{\omega}))=\exp ( S_{i-1}f(\omega\tau_{\omega}))+f(\omega_i...\omega_n\tau_{\omega})+S_{n-i}f(\omega_{i+1}...\omega_n\tau_{\omega}))$$
$$\leq Q.c_N.Q=Q^2c_N.$$
Therefore for small enough $\epsilon>0$ we have
$$\|\mathcal{L}_s ^n(\mathcal{E}_n g-\mathcal{E}_{n,N}g)\|_{\alpha}$$
$$\leq  \| g\|_{\infty}(2+|s|K)Q^2 \sum_{\omega  \in E_A^n(N+)}\exp \left(\epsilon  S_nf(\omega \tau_{\omega})\right) \exp \left((x-\epsilon) S_nf(\omega \tau_{\omega})\right) $$
$$ \leq \| g\|_{\infty}(2+|s|K)Q^2Q^{2\epsilon}c_N^{\epsilon} \sum_{\omega  \in E_A^n(N+)} \exp \left((x-\epsilon) S_nf(\omega \tau_{\omega})\right) $$
$$\leq \| g\|_{\infty}(2+|s|K)Q^4c_N^{\epsilon} \sum_{\tau \in F} \mathcal{L}_{x-\epsilon} ^n(\mathbb{1})(\tau)\leq  \| g\|_{\infty}(2+|s|K)Q^4c_N^{\epsilon} \#F\| \mathcal{L}_{x-\epsilon} ^n\|_{\alpha}.$$
This together with \ref{lasota} yields
$$\|\mathcal{L}_s ^n(\mathcal{E}_n -\mathcal{E}_{n,N})\|_{\alpha} \leq (2+|s|K)Q^4c_N^{\epsilon} \#F\| \mathcal{L}_{x-\epsilon} ^n\|_{\alpha}$$
$$ \leq (2+|s|K)Q^4c_N^{\epsilon} \#F(Q+C) e^{nP(x-\epsilon)}.$$
For large enough $N$ we get
$$\|\mathcal{L}_s ^n(\mathcal{E}_n -\mathcal{E}_{n,N})\|_{\alpha} \leq e^{nP(x-\epsilon)}e^{-\alpha n}.$$
Thus since $P$ is strictly decreasing, the above inequality combined with \ref{triangle L} and \ref{approximating L(I-E)} implies
$$\|\mathcal{L}_s ^n-\mathcal{L}_s ^n \mathcal{E}_{n,N} \|_{\alpha} \leq C_1e^{nP(x)}e^{-\alpha n} + e^{nP(x-\epsilon)}e^{-\alpha n}\leq C_2 e^{nP(x-\epsilon)}e^{-\alpha n}.$$
Therefore, we can estimate the essential spectral radius by proposition \ref{ess spec rad}:
$$r_{\text{ess}}(\mathcal{L}_s)=\lim_n \|\mathcal{L}_s^n\|_{\mathfrak{K}}^{1/n} \leq \limsup_n \|\mathcal{L}_s ^n-\mathcal{L}_s ^n \mathcal{E}_{n,N} \|_{\alpha}^{1/n}\leq e^{P(x-\epsilon)}e^{-\alpha }.$$
Since $\epsilon$ was chosen small enough, this completes the proof.
\end{proof}
We want to introduce two operators closely related to the transfer operator.
The first operator is $\mathcal{L}_{0}$. There is $s$ hidden in the definition but we don't write that. It is defined by:
$$\mathcal{L}_0:=e^{-P(x)}\mathcal{L}_s,$$
and another operator is the weighted operator defined by:
$$\bar{\mathcal{L}}_sg:=e^{-P(x)}\dfrac{1}{h_x}\mathcal{L}_s(gh_x),$$
where $h_{x}$ is a fixed point of $\mathcal{L}_0$ obtained in \citep[p. 34]{MU} as the (compactly) convergent point of the sequence $\{ \frac{1}{n_k}\sum_{i=0}^{n_k-1}e^{-P(x)}\mathcal{L}_{x}^j(\mathbb{1})\}$. In other words, $h_x$ is actually an eigenfunction of $\mathcal{L}_x$ corresponding to the eigenvalue $e^{P(x)}$. Moreover, it is clear that $\int h_x \text{d}m_x=1.$
\begin{lemma}\label{epsilon bounded below}
There is $c>0$ such that $h_x>c$.
\end{lemma}
\begin{proof}
We use theorem 2.3.5 from \citep[p. 29]{MU} to show this. Let $n_k-1=(M+1)t_k+r_k$ where $0\leq r_k \leq M$ then
$$\frac{1}{n_k}\sum_{i=0}^{n_k-1}\mathcal{L}_0^j(\mathbb{1}) \geq \frac{1}{n_k}\sum_{i=1}^{(M+1)t_k}\mathcal{L}_0^j(\mathbb{1}) \geq \frac{1}{n_k}t_k R$$
which leaves $h_{x}\geq \frac{R}{M+1}$.
\end{proof}
\begin{lemma}\label{constant modulus}
If $g \in C^{0,\alpha}(E_A^{\mathbb{N}},\mathbb{C})$ is non-negative then $\{ \frac{1}{n}\sum_{j=1}^n\bar{\mathcal{L}}_{x}^jg\}$ has a converging subsequnce with  limit $ \int g\text{d}\mu_x$, where $\mu_x$ is the equilibrium state of $xf$.
\end{lemma}
\begin{proof}
Observe that $\bar{\mathcal{L}}_{x}(\mathbb{1})=\mathbb{1}$ and so $\bar{\mathcal{L}}_{x}^j(\mathbb{1})=\mathbb{1}$ for each $j\geq1$. Then one can start with $\|\bar{\mathcal{L}}_{x}^jg\|_{\alpha} \leq \|g\|_{\alpha}$ and follow the same proof of theorem 2.4.3 \citep[p. 34]{MU} to find that $\{ \frac{1}{n}\sum_{j=1}^n\bar{\mathcal{L}}_{x}^jg\}$  has a converging subsequence with limit $g_1 \in C^{0,\alpha}(E_A^{\mathbb{N}},\mathbb{C})$, where $\bar{\mathcal{L}}_{x}g_1=g_1$. This leaves $g_1h_{x}$ as a fixed point of $\mathcal{L}_0$. Since $g$ is non-negative so is $g_1$ and $g_1h_{x}$. Now theorem 2.4.7 \citep[p. 39]{MU} tells us that 
$$(\frac{g_1}{d}h_{x}m_x)\circ\sigma^{-1}=\frac{g_1}{d}h_{x}m_x, \; \; \; d=\int g_1h_{x}dm_x,$$ 
where $m_x$ is eigenmeasure of $\mathcal{L}_x$. Therefore if one defines a measure by $\mu_1(A)=\frac{1}{d} \int_Ag_1h_{x}dm_x$, we find that 
$$\mu_1(\sigma^{-1}(A))=\frac{1}{d}\int_{\sigma^{-1}(A)}g_1h_{x}dm_x$$
$$=\frac{1}{d}\int_Ag_1\circ\sigma^{-1}h_{x}\circ\sigma^{-1}d(m\circ\sigma^{-1})=\frac{1}{d}\int_A g_1h_{x}dm_x=\mu_1(A).$$
That leaves an invariant absolutely continuous measure with respect to $m_x$. Then theorem 10.4.2 \citep{Munday} implies that $\mu_1$ must be $\mu_x$, therefore the Randon-Nikodym derivative of $\mu_1$ with respect to $m$ is the same as that of $\mu$ with respect to $m$ a.e. which means $ g_1=d$ a.e. and  since $g_1$ is continuous so $g_1=d=\int g_1h_{x}dm_x=\int g_1d\mu_x$ everywhere. Furthermore, it is not hard to see that $(\bar{\mathcal{L}}_x)^*(\mu_x)=\mu_x$ see theorem 2.4.4 \citep[~ p. 36]{MU}. Since we had
 $$\frac{1}{n}\sum_{j=1}^n\bar{\mathcal{L}}_{x}^jg \rightarrow g_1$$
 on a sub-sequence, then 
 $$\int g\text{d}\mu_x=\int \frac{1}{n}\sum_{j=1}^n\bar{\mathcal{L}}_{x}^jg \text{d}\mu_x \rightarrow \int g_1 \text{d}\mu_x,$$
 i.e. $\int g\text{d}\mu_x=\int g_1 \text{d}\mu_x$.
\end{proof}
\begin{proposition}\label{jordan}
The transfer operator $\mathcal{L}_s$ has at most finitely many eigenvalues of modulus $e^{P(x)}$ all of which with multiplicity one.
\end{proposition}
\begin{proof}
The previous proposition implies there are at most finitely many spectral values of $\mathcal{L}_s$ with modulus $e^{P(x)}$ all are isolated eigenvalues with finite (algebraic) multiplicity, see definition \ref{essential spectrum def}.  We would like to show first for each eigenvalues $\lambda$ with $|\lambda|=e^{P(x)}$ the transfer operator $\mathcal{L}_s$ acts on $X:=\mathcal{P}_{\lambda,s}(C^{0,\alpha}(E_A^{\mathbb{N}},\mathbb{C}))=\cup_{m\geq 1} \ker(\mathcal{L}_s - \lambda)^m$ diagonally. To see this we consider the Jordan normal form of $L:=\mathcal{L}_s$ on finite-dimensional space $X$, so there is an invertible transformation $P$, such that $PLP^{-1}$ is the Jordan normal form of $L$. Consider a $k\times k$ Jordan block in matrix representation that has $1$ above the diagonal. The $n^{th}$ power of the block looks like
\[
\begin{bmatrix}
    \lambda^n & \binom{n}{1} \lambda^{n-1} & \binom{n}{2}   \lambda^{n-2} & \dots  & \binom{n}{k-1} \lambda^{n-k+1} \\
              & \lambda^n & \binom{n}{1} \lambda^{n-1} & \dots  & \binom{n}{k-2} \lambda^{n-k+2} \\
          &           & \lambda^n  & \dots  & \binom{n}{k-3} \lambda^{n-k+3}  \\
           &   &   & \ddots & \vdots \\
       &   &   &    & \lambda^n
\end{bmatrix}.
\]
Then for 
$e=\begin{bmatrix}
0 & 0 & ... & 0 & 1
\end{bmatrix}^T$ 
we have
\[
\begin{bmatrix}
    \lambda^n & \binom{n}{1} \lambda^{n-1} & \binom{n}{2}   \lambda^{n-2} & \dots  & \binom{n}{k-1} \lambda^{n-k+1} \\
              & \lambda^n & \binom{n}{1} \lambda^{n-1} & \dots  & \binom{n}{k-2} \lambda^{n-k+2} \\
          &           & \lambda^n  & \dots  & \binom{n}{k-3} \lambda^{n-k+3}  \\
           &   &   & \ddots & \vdots \\
       &   &   &    & \lambda^n
\end{bmatrix}e
=
\begin{bmatrix}
 \binom{n}{k-1} \lambda^{n-k+1}\\
 \binom{n}{k-2} \lambda^{n-k+2}\\
 \binom{n}{k-3} \lambda^{n-k+3}\\
 \vdots\\
 \lambda^n
\end{bmatrix}.
\]
Notice that $\binom{n}{1} \lambda^{n-1}$ is the $(k-1)^{th}$ coordinate of this vector. If we equip $X$ with the norm
$$\| x\|=|x_1|+...+|x_t|, \; \; \; t=\dim X, $$
and if we view $e$ and the the above vector in $X$, we will have:
$$\binom{n}{1} |\lambda|^{n-1} \leq \| PL^nP^{-1}e\| \leq \|PL^nP^{-1}\| \leq \|P\| \| L^n\| \|P^{-1}\| \leq C_0 |\lambda|^n$$
for some constant $C_0$, where the last inequality holds by proposition \ref{spectral radius L} and because on finite-dimensional space all the norms are equivalent. This is clearly a contradiction. Therefore there is no non-trivial Jordan block, i.e. $L$ is diagonalizable. This implies 
$$X=\ker (\mathcal{L}_s-\lambda).$$
It is clear that if $g$ is in $\ker(\mathcal{L}_s-\lambda)$ then $g/h_{x}$ is in $\ker(\bar{\mathcal{L}}_s-e^{-P(x)}\lambda)$. Therefore to show each $\ker(\mathcal{L}_s-\lambda)$ is one dimensional, it is enough to show $\ker(\bar{\mathcal{L}}_s-e^{-P(x)}\lambda)$ is one dimensional. Let $g \in \ker(\bar{\mathcal{L}}_s-e^{-P(x)}\lambda)$, for each $n$ 
$$|g|=|e^{-P(x)}\lambda g|=|\bar{\mathcal{L}}_s^ng|\leq \bar{\mathcal{L}}_{x}^n|g|.$$
Therefore if we apply the above lemma to the function $|g|$ we obtain
$$|g|\leq \int |g| d\mu_x.$$
Continuity of $g$ and the fact that $\text{supp}(\mu_{x})=E_A^{\mathbb{N}}$ (see explanation below the definition \ref{def Gibbs}) makes this inequality into equality, i.e. every eigenvector has constant modulus. It is not hard to see that 
$$\bar{\mathcal{L}}_s^ng(\rho)=$$
$$\frac{e^{-nP(x)}}{h_{x}(\rho)}\sum_{\omega \in E_A^n}\exp(sS_nf(\omega \rho))\dfrac{1}{h_{x}(\sigma^{n-1}\omega\rho)}\dfrac{1}{h_{x}(\sigma^{n-2}\omega\rho)}...\dfrac{1}{h_{x}(\sigma\omega\rho)}h_{x}(\omega\rho)g(\omega\rho).$$
Moreover since $\bar{\mathcal{L}}_{x}(\mathbb{1})=\mathbb{1}$ we get:
$$1=\bar{\mathcal{L}}_{x}^n(\mathbb{1})(\rho)=$$
$$\sum_{\omega \in E_A^n}e^{-nP(x)}\frac{1}{h_{x}(\rho)}\exp(x S_nf(\omega \rho))\dfrac{1}{h_{x}(\sigma^{n-1}\omega\rho)}\dfrac{1}{h_{x}(\sigma^{n-2}\omega\rho)}...\dfrac{1}{h_{x}(\sigma\omega\rho)}h_{x}(\omega\rho).$$
Note that every term in this sum, say $u_{\omega}$, is positive. Eventually, we find:
$$e^{-nP(x)}\lambda^ng(\rho)=\bar{\mathcal{L}}_s^ng(\rho)= \sum_{\omega \in E_A^n}u_{\omega}\exp(iy S_nf(\omega \rho))g(\omega \rho).$$
Now note that $|\sum_j a_j|=\sum_j|a_j|$ implies all $a_j$ are co-linear, this along with the fact that $g$ has constant modulus we get 
$$g(\omega \rho)=e^{-nP(x)}\lambda^n\exp(-iy S_nf(\omega \rho))g(\rho). $$
This means values of $g$ on the dense set $\cup_n\sigma^{-n}(\rho)$ (see remark below the definition \ref{finitely irreducible}) is determined by $g(\rho)$, so $g$ spans $\ker(\bar{\mathcal{L}}_s-e^{-P(x)}\lambda)$ as long as $g$ has at least one non-zero point. This shows $\lambda$ is a simple eigenvalue and it finishes the proof.
\end{proof}
Thus everything is ready to obtain spectral representation of $\mathcal{L}_s$ corresponding to the eigenvalues $\lambda_1, \lambda_2, ...,  \lambda_p$ of modulus $e^{P(x)}$. We use the above proposition to see that for each $s=x+iy \in \Gamma^+$, $\mathcal{L}_s$ has only finitely many eigenvalues $\lambda_1(s),...\lambda_n(s)$ of modulus $e^{P(x)}$ each of which isolated in the spectrum and actually they are all simple eigenvalues. Therefore we may use theorem \ref{kato theorem} to obtain the following spectral representation of the transfer operator:
$$ \mathcal{L}_s=\lambda_1(s)\mathcal{P}_{1,s}+ \lambda_2(s)\mathcal{P}_{2,s} + ...+ \lambda_n(s)\mathcal{P}_{n,s}+\mathcal{D}_{s},$$
where each $\mathcal{P}_{i,s}$ is projection. Note that in this equation the operators are analytic operators and eigenvalues are analytic functions. Further, the composition of every two different operators on the right-hand side vanishes by proposition \ref{spectral mapping theorem}. This yields 
\begin{equation}\label{spectral decomposition sum}
\mathcal{L}_s^n=\lambda_1(s)^n\mathcal{P}_{1,s}+ \lambda_2(s)^n\mathcal{P}_{2,s} + ...+ \lambda_n(s)^n\mathcal{P}_{n,s}+\mathcal{D}_{s}^n.
\end{equation} 
Finally proposition \ref{spectral mapping theorem} implies:
$$\text{Sp}(\mathcal{L}_s)\cup\{0\}=\{\lambda_1(s)\}\cup \{\lambda_2(s)\}\cup ... \cup\{\lambda_n(s)\}\cup\text{Sp}(\mathcal{D}_s)\cup\{0\}.$$
We finish this section with the following lemma.
\begin{lemma}\label{spectral D^m}
For every $s_0$ on the line $x=1$, there is a neighborhood $U$ of $s_0$, $0<\beta<1$ and constant $C>0$ such that for every positive integer $m$
$$\|\mathcal{D}_{s}^m\|_{\alpha} \leq C \beta^m, \; \; \; s \in U.$$
\end{lemma}
\begin{proof}
The above spectral decomposition implies the spectral radius of $\mathcal{D}_s$ to be strictly less than that of $\mathcal{L}_s$. Furthermore, proposition \ref{spectral radius L} implies $r(\mathcal{L}_{s_0})\leq e^{P(1)}=1$ so for $s_0$ there is $0<\beta<1$ such that $r(\mathcal{D}_{s_0})<\beta.$
Thus there is constant $C_1$ and natural number $q$ such that $$\|\mathcal{D}_{s_0}^q\|_{\alpha}\leq C_1\beta^q \leq \frac{\beta}{2}.$$ 
Additionally, using continuity on a small enough ball $U$ at $s_0$  we have 
$$\|\mathcal{D}_{s}^q-\mathcal{D}_{s_0}^q\|_{\alpha}<\frac{\beta}{2}.$$  
Combining these two recent inequalities yields  $\|\mathcal{D}_{s}^q\|_{\alpha}\leq \beta$ on $U$. Furthermore, there is constant $C_2$ such that for each integer r with $0\leq r <q$,  we have $\|\mathcal{D}_{s}^r\|_{\alpha}\leq C_2$ on  $U$. Since for each positive integer $m$ we can write $m=lq+r$, we eventually get for some $C>0$:
\begin{equation}\label{spectral radius tend to 0}
\|\mathcal{D}_{s}^m\|_{\alpha}\leq C(\beta^{1/q})^m
\end{equation} 
on $U$.
\end{proof}
\section{Counting Function and Poincar\'{e} series}
Given $\rho \in E_A^{\mathbb{N}}$ and $B \subseteq E_A^{\mathbb{N}}$, for every $T>0$ we define several \textbf{counting functions}. \\
(a) The central counting function for us is
\begin{equation}\label{eq1}
N_{\rho}(B,T):=\#\{ \omega \in \cup_{n=1}^{\infty} E_A^n \; : \; \omega\rho \text{ admissible},\; \omega\rho \in B, \;  S_{|\omega|}f(\omega\rho)\geq - T \}
\end{equation}
It is not so hard to see that this is a step function of $T$, continuous from right and increasing. In order to associate a complex function to this counting function we set $N_{\rho}(B,T)=0$ for $T<0$ and we consider the Laplace–Stieltjes transform of $T \mapsto N_{\rho}(B,T)$ which we call it \textbf{Poincar\'{e} series}:
$$\eta_{\rho}(B,s):=\int _0 ^{\infty} \exp (-sT)\text{d}N_{\rho}(B,T).$$
We will talk about its convergence in the next proposition. Below we introduce other counting functions appropriate for our purposes.\\
(b) Let $H=\{\tau_{(i)}\}_{i \in I}$ be a countable (finite or infinite) collection of finite words of bounded length, i.e. there exists a positive integer $k$ such that $|\tau_{(i)}|\leq k$ for each $i\in I$. Further, assume the cylinders $\{[\tau_{(i)}]\}_{i \in I}$ are mutually disjoint. We denote 
$$[H]:=\cup_{i \in I} [\tau_{(i)}],$$ 
then the corresponding Poincar\'{e} series is of the form
$$\eta_{\rho}([H],s)=\int _0 ^{\infty} \exp(-sT) \text{d}N_{\rho}([H],T). $$
$$=\sum _{n=1} ^{\infty} \exp(-sT_i)\left(N_{\rho}([H],T_i)-N_{\rho}([H],T_{i-1})\right)$$
where $T_1 < T_2 < T_3 <...$ is the increasing sequence of discontinuities of $T \mapsto N_{\rho}([H],T)$. Eventually this sums up to 
\begin{equation}\label{eta}
\eta_{\rho}([H],s)= \sum_{n=1} ^{\infty} \sum\limits_{\substack{ \omega \rho \in [H]}}\exp(sS_nf(\omega \rho))=\sum_{n=1} ^{\infty} \mathcal{L}_s ^n (\mathbb{1}_{[H]})(\rho).
\end{equation}
(c) If we require to count only words with certain initial blocks then we should define
$$ N_{\rho}(H,T):=\#\{\omega \in \cup_{n=1}^{\infty} E_A^n : \tau \in H, \; \tau\omega\rho \; \text{ admissible },  \; S_{|\tau\omega|}f(\tau \omega \rho) \geq - T \}.$$
Then similarly one can see that the corresponding Poincar\'{e} series has the form
$$\eta_{\rho}(H,s)=\sum_{n=1}^{\infty}\mathcal{L}_s ^{k+n} (\mathbb{1}_{[H]})(\rho).$$
Therefore
\begin{equation}\label{eta relation}
\eta_{\rho}([H],s)=\eta_{\rho}(H,s)+\sum_{n=1}^{k}\mathcal{L}_s ^{n} (\mathbb{1}_{[H]})(\rho).
\end{equation}
(d) For any positive integer $q$ we set
$$N_{\rho}([H],q,T):=\#\{\omega \in E_A^q : \omega\rho \; \text{ admissible}, \; \omega\rho \in [H], \;  S_{|\omega|}f(\omega\rho)\geq-T\},$$
then its Poincar\'{e} series would be 
$$\eta_{\rho}([H],q,s)=\mathcal{L}_s^q(\mathbb{1}_{[H]})(\rho).$$
(e) Further we would like to deal with periodic words as well. For this purpose we define
$$N_{\text{per}}([H], T):=\#\{\omega \in \cup_{n=1}^{\infty} E_A^n \;  : \; \omega\;  \text{periodic word},  \; \overline{\omega} \in [H], \; S_{|\omega|}f(\overline{\omega})\geq -T \},$$
(f) And
$$N_{\text{per}}(H, T):=\#\{\omega \; : \tau \in H, \; \; \tau \omega\;  \text{periodic word}, \; S_{|\tau \omega|}f(\overline{\tau \omega})\geq -T \}.$$
(g) Finally we introduce another counting function for any positive integer $q$:
$$N_{\text{per}}([H], q, T):=\#\{\omega : \omega\;  \text{periodic word of length $q$},  \; \overline{\omega} \in [H], \; S_{|\omega|}f(\overline{\omega})\geq -T \}.$$
(h) If $H=E$ then obviously $[H]$ is the whole space $E_A^{\mathbb{N}}$. In this case, we drop the notation $E_A^{\mathbb{N}}$ in $N_{\rho}(E_A^{\mathbb{N}},T)$ and we simply write $N_{\rho}(T)$, similarly $N_{\text{per}}(T)$.\\
Next, we want to find some relations between these counting functions. Note that we do not introduce a Poincar\'{e} series for the periodic orbits, as it won't have an ordinary geometric series expression, and therefore Tauberian theorems are not applicable. Instead, we use some approximations. Now for every finite word $\omega$ we pick (exactly) one $\omega^+ \in E_A^{\mathbb{N}}$ such that $\omega\omega^+$ is admissible. From now on in this section, we assume $\tau$ is a fixed word of length $k\geq 0$. When $k=0$ we mean there is no word involved.
\begin{lemma}
Let $q$ be a positive integer, and $\gamma \in E_A^q$ be any word of length $q$. Given any $\omega$ where $\tau\gamma\omega$ is admissible and  it is further a periodic word, then we have
$$|S_{|\tau\gamma\omega|}f(\overline{\tau\gamma\omega})-S_{|\tau\gamma\omega|}f(\tau\gamma\omega\tau\gamma(\tau\gamma)^+)|\leq Ke^{-(k+q)\alpha},$$
where $K$ only depends on $f$.
\end{lemma}
\begin{proof}
It is enough to apply Lemma \ref{bounded S_n}:
$$|S_{|\tau\gamma\omega|}f(\overline{\tau\gamma\omega})-S_{|\tau\gamma\omega|}f(\tau\gamma\omega\tau\gamma(\tau\gamma)^+)|\leq K_{f}d(\overline{\tau\gamma\omega}, \tau\gamma(\tau\gamma)^+)^{\alpha} \leq K_{f}e^{-(k+q)\alpha}.$$
\end{proof}

\begin{lemma}\label{bound for periodic}
Let $q$ be a positive integer, then the following inequalities hold:
\begin{itemize}
    \item[(i)] $$ N_{\text{per}}([\tau],q, T) \leq N_{\tau \tau^+}([\tau],q, T+K),$$
    \item[(ii)] $$\sum\limits_{\substack{\gamma \in E_A^q \\ \tau \gamma \in E_A^{k+q}}} N_{\tau\gamma(\tau\gamma)^+}(\tau\gamma, T-Ke^{-(k+q)\alpha}) \leq N_{\text{per}}(\tau,T),$$
    \item[(iii)] $$N_{\text{per}}([\tau],T) \leq  \sum\limits_{\substack{\gamma \in E_A^q \\ \tau\gamma \in E_A^{k+q}}} N_{\tau\gamma(\tau\gamma)^+}([\tau\gamma], T+Ke^{-(k+q)\alpha}),$$
    \item[(iv)] For $i\geq k+q$ 
    $$N_{\tau\gamma(\tau\gamma)^+}([\tau\gamma],i,T) \leq N_{\tau\tau^+}([\tau\gamma],i,T+K),$$
    \item[(v)] If $F$ is any finite subset of $E^q_A$ and $F'=E^q_A \setminus F$, then
    $$N_{\text{per}}([\tau],T) \leq   \sum\limits_{\substack{ \gamma \in F \\ \tau\gamma \in E_A^{k+q}}} N_{\tau\gamma(\tau\gamma)^+}([\tau\gamma], T+Ke^{-(k+q)\alpha})$$
    $$+ \sum\limits_{\substack{ \gamma \in F' \\ \tau\gamma \in E_A^{k+q}}} N_{\tau\tau^+}([\tau\gamma], T+2K)+\sum_{i=1}^{k+q-1} N_{\tau\tau^+}([\tau],i,T+K)$$
\end{itemize}
where $K$ only depends on $f$.
\end{lemma}
\begin{proof}
(i) Let $\omega$ be a finite word contributing to $N_{per}([\tau],q,T)$, then $|\omega|=q$. The fact that $\overline{\omega} \in [\tau]$ gives $\omega_1=\tau_1$ Therefore since $\omega_q\omega_1$ is admissible, so is $\omega\tau\tau^+$. If $q\geq k=|\tau|$, since $\overline{\omega} \in [\tau]$ so is $\omega\tau\tau^+$, and if $q<k$, since $\overline{\omega} \in [\tau]$, we can write $\tau$ as $m$ copies of $\omega$ and some remainders, i.e. $\tau=\omega^m\omega_1...\omega_r$. It is clear then the first $k$ letters of $\omega^{m+1}\omega_1...\omega_r$ is again $\tau$. Thus $\omega\tau\tau^+ \in [\tau]$. It remains to show $S_{|\omega|}\psi(\omega\tau\tau^+)\geq-T-K$. From our assumption $S_{|\omega|}\psi(\overline{\omega})\geq-T$, we can apply Lemma \ref{bounded S_n} to see that 
$$S_{|\omega|}\psi(\overline{\omega})\leq S_{|\omega|}\psi(\omega\tau\tau^+)+K.$$
This finishes the proof for part (i).\\
(ii) Let $\gamma$ be a word of length $q$ with $\tau\gamma$ admissible. Let $\omega$ be a finite word contributing to $N_{\tau\gamma(\tau\gamma)^+}(\tau\gamma, T-Ke^{-(k+q)\alpha})$, we want to show $\gamma\omega$ contributes to $N_{per}(\tau,T)$. It is clear that this way of contribution is injective, which proves (ii). Since $\tau\gamma\omega\tau\gamma(\tau\gamma)^+$ is admissible, so is $\overline{\tau\gamma\omega}$. Furthermore, we know $S_{|\tau\gamma\omega|}\psi(\tau\gamma\omega\tau\gamma(\tau\gamma)^+)\geq -T+Ke^{-(k+q)\alpha}$. If we use the above lemma we find 
$$-T+Ke^{-(k+q)\alpha}\leq S_{|\tau\gamma\omega|}\psi(\tau\gamma\omega\tau\gamma(\tau\gamma)^+) \leq S_{|\tau\gamma\omega|}\psi(\overline{\tau\gamma\omega})+Ke^{-(k+q)\alpha},$$
which shows $-T \leq S_{|\tau\gamma\omega|}\psi(\overline{\tau\gamma\omega})$ as needed.\\
(iii). Let $\omega$ be a finite word contributing to $N_{per}([\tau],T)$ of length $n$. The fact that $\overline{\omega} \in [\tau]$ gives $\omega_1=\tau_1$ Therefore since $\omega_n\omega_1$ is admissible, so is $\omega\tau$. Note that 
$$[\tau]=\cup'[\tau\gamma],$$
where the union is over all $\gamma$ with length $q$ such that $\tau\gamma$ is admissible. Since $\overline{\omega} \in [\tau]$, there should be $\gamma$ such that $\overline{\omega} \in [\tau\gamma]$. Since $\omega\tau$ is admissible, so is $\omega\tau\gamma(\tau\gamma)^+$. Next we want to show $\omega\tau\gamma(\tau\gamma)^+ \in [\tau\gamma]$. If we separate into two cases where $n\geq k+q$ and $n<k+q$, then in exactly a similar manner as in part (i) we obtain this. It remains only to show $S_{|\omega|}\psi(\omega\tau\gamma(\tau\gamma)^+)\geq -T-Ke^{-(k+q)\alpha}$. We have already $S_{|\omega|}\psi(\overline{\omega})\geq -T$, furthermore if we use lemma \ref{bounded S_n} we see that
$$|S_{|\omega|}\psi(\overline{\omega})-S_{|\omega|}\psi(\omega\tau\gamma(\tau\gamma)^+)|\leq K d(\overline{\omega},\tau\gamma(\tau\gamma)^+)\leq K e^{-(k+q)\alpha},$$
where the last inequality is due to $\overline{\omega} \in [\tau\gamma]$. Thus from this inequality, we obtain 
$$-T-Ke^{-(k+q)\alpha} \leq S_{|\omega|}\psi(\omega\tau\gamma(\tau\gamma)^+).$$ This completes part (iii).\\
(iv) Take $\omega$ that contributes  to $N_{\tau\gamma(\tau\gamma)^+}([\tau\gamma],t, T)$. Clearly, $\omega\tau\tau^+$ is admissible. Since $|\omega|\geq k+q$ then we have clearly $\omega\tau\tau^+ \in [\tau\gamma]$ as well. Further, note that 
$$|S_{|\omega|}(\omega\tau\gamma(\tau\gamma)^+)-S_{|\omega|}(\omega\tau\tau^+)|\leq K.$$
(v) Take $\omega$ such that it contributes to $N_{per}([\tau],T)$. If its length is less than $k+q$, then we use part (i). This contributes to the third sum on the right-hand side. If the length of $\omega$ is at least $k+q$, then (iii) and (iv) tell us $\omega$ contributes to either of the first two sums on the right-hand side. This finishes the proof of (v) and the lemma.
\end{proof}
Below we want to prove the item (v) from the above lemma without $[\tau]$. Let $\rho \in E_A^{\mathbb{N}}$, then due to our assumption that shift space is finitely irreducible, there exists a finite set consisting of finite words
$$\Omega=\{\tau_{(1)},...,\tau_{(r)} \}$$
such that for every finite word $\omega$ there exists $\tau_{(j)} \in \Omega$ with $\omega\tau_{(j)}\rho$ being admissible. Below we have a summation over all $\tau_{(j)}\rho$, while this might not be admissible for all $j=1,...,r$. Note that the sum is only taken over those $j$s where $\tau_{(j)}\rho$ is admissible. Note, that this $\Omega$ and $r$ is independent of $\rho$.
\begin{lemma}\label{bound for periodic whole}
If $F$ is any finite subset of $E^q_A$ and $F'=E^q_A \setminus F$, for any $\rho \in E_A^{\mathbb{N}}$ we have
    $$N_{\text{per}}(T) \leq   \sum\limits_{\substack{ \gamma \in F }} N_{\gamma\gamma^+}([\gamma], T+Ke^{-q\alpha})$$
    $$+ \sum\limits_{\substack{ j=1 \\ \tau_{(j)}\rho \in E_A^{\mathbb{N}}}}^r N_{\tau_{(j)}\rho}([F'], T+K)+\sum\limits_{\substack{ j=1 \\ \tau_{(j)}\rho \in E_A^{\mathbb{N}}}}^r  \sum_{i=1}^{q-1} N_{\tau_{(j)}\rho}(i,T+K),$$
    where $K$ only depends on $f$.
\end{lemma}
\begin{proof}
The proof is similar to item (v) in the above lemma. Let $\omega$ be a finite word contributing to $N_{\text{per}}(T)$, pick $\tau_{(j)} \in \Omega$ such that $\omega\tau_{(j)}\rho$ is admissible. If $|\omega|<q$, clearly $\omega$ is contributing to the third term on the right hand of the inequality. If $|\omega| \geq q$ and $\omega_1...\omega_q \in F$, we want to show $\omega$ contributes to $N_{\gamma\gamma^+}([\gamma], T+Ke^{-q\alpha})$ where $\gamma=\omega_1...\omega_q$. Since $\omega$ is periodic $\omega\omega_1$ is admissible, and so is $\omega\gamma\gamma^+$. It is clear that $\omega\gamma\gamma^+ \in [\gamma]$ as well. Further, note that 
$$|S_{|\omega|}f(\omega\gamma\gamma^+)-S_{|\omega|}f(\overline{\omega})|\leq Kd(\gamma\gamma^+,\overline{\omega})\leq Ke^{-q\alpha}.$$
Finally, in case $|\omega|\geq q$ and $\omega_1...\omega_q \in F'$ we want to show $\omega$ is contributing to the second sum on the right-hand side. This is similar to our previous case.
\end{proof}
Moreover, we have the following two estimates for the eigenfunction $h$ and the equilibrium state $\mu$.
\begin{lemma}
Let $\omega$ be a word of length $n$ such that $\omega\rho$, $\omega\rho'$ are admissible, then we have
$$1-K_{1}e^{-n\alpha}\leq \frac{h(\omega\rho)}{h(\omega\rho')}\leq 1+K_{1}e^{-n\alpha},$$
where $K_{1}$ only depends on $h$.
\end{lemma}
\begin{proof}
We know from \citep[~ p. 34]{MU} that $h$ is H\"{o}lder continuous, therefore there is a constant $K_0$, such that
$$|h(\omega\rho)-h(\omega\rho)|\leq K_{0}d(\omega\rho,\omega\rho')\leq K_{0}e^{-n\alpha}.$$
Dividing by $h(\omega\rho)$ and using lemma \ref{epsilon bounded below}, we obtain
$$|\frac{h(\omega\rho)}{h(\omega\rho')}-1|\leq \frac{K_0}{h(\omega\rho)}e^{-n\alpha}\leq K_{1}e^{-n\alpha},$$
where $K_{1}=K_0\frac{M+1}{R}$.
\end{proof}
\begin{lemma}\label{equivalent measure}
Let $\omega$ be a finite word of length $n$ such that $\omega\rho$ is admissible, then
$$(1-K_{1}e^{-n\alpha})h(\omega\rho)m([\omega])\leq \mu([\omega])\leq (1+K_{1}e^{-n\alpha})h(\omega\rho)m([\omega]),$$
where $K_1$ is a constant depending only on $h$.
\end{lemma}
\begin{proof}
We saw in the proof of the lemma \ref{constant modulus} that $\mu(A)=\int_A h \text{d}m$. Therefore we have
$$\left(\inf_{[\omega]} h\right) m([\omega])\leq \mu([\omega])\leq \left(\sup_{[\omega]} h\right) m([\omega]).$$
Now we use the above lemma to see
$$(1-K_{1}e^{-n\alpha})h(\omega\rho)\leq \inf_{[\omega]}  h \leq \sup_{[\omega]} h \leq (1+K_{1}e^{-n\alpha})h(\omega\rho).$$
This finishes the proof.
\end{proof}
\begin{proposition}\label{eta holomorphic}
The functions $\eta_{\rho}([H],s)$, $\eta_{\rho}(H,s)$ are holomorphic on $\operatorname{Re}(s) >1$, and the function $\eta_{\rho}([H],q,s)$ is holomorphic on $\Gamma^+$.
\end{proposition}
\begin{proof}
Using the relation \ref{eta relation}, if we show $\eta_{\rho}([H],s)$ is holomorphic then $\eta_{\rho}(H,s)$ will be holomorphic as well. In order to show $\eta_{\rho}([H],s)$ is holomorphic we need $| \mathcal{L}_s ^n (\mathbb{1}_{[H]})|_{\infty}$:
$$ |\mathcal{L}_s ^n (\mathbb{1}_{[H]})|_{\infty}\leq |\mathcal{L}_s ^n (\mathbb{1})|_{\infty} \leq \sum\limits_{\substack{ \omega \in E_A^n }}\exp(\operatorname{Re}(s) \sup _{[\omega]}S_nf).$$
This reminds us of the pressure function. Using the fact that $P$ is strictly decreasing on $\Gamma$ from proposition \ref{pressure decreasing}, consider an arbitrary $s_0=x_0+iy_0$ with $x_0>1$, for any $s$ with $x\geq x_0$ there is a negative $r$ such that $P(x)<r<0$, therefore there is $N$ such that for $n>N$:
$$\frac{1}{n}\ln \big( \sum\limits_{\omega \in E_A^n}\exp(x \sup_{[\omega]}S_nf \big) < r,$$
so
$$ |\mathcal{L}_s ^n (\mathbb{1}_{[H]})|_{\infty}\leq |\mathcal{L}_s ^n (\mathbb{1})|_{\infty}\leq \sum\limits_{\omega \in E_A^n}\exp(x \sup _{[\omega]}S_nf) < e^{rn}.$$
This shows $\eta_{\rho}([H],s)$ converges uniformly on compact sets, thus $\eta_{\rho}([H],s)$ as a sum of holomorphic functions is holomorphic on $\operatorname{Re}(s)>1$.\\
The above expression of $\eta_{\rho}([H],q,s)$ shows it is holomorphic on $\Gamma^+.$
\end{proof}
\begin{proposition}\label{residue}
If $f: E_A^{\mathbb{N}} \rightarrow \mathbb{R}$ has D-generic property, then each $\eta_{\rho}([H],s)$ and  $\eta_{\rho}(H,s)$ at each  point of the critical line $\operatorname{Re}(s) =1$ except $s=1$ admits analytic continuation and at $s=1$ admits a meromorphic extension with a simple pole and residue
$$\text{Res}(\eta_{\rho}, 1)=\frac{h(\rho)}{\chi_{\mu}}m([H]).$$
If we lift the D-generic property, then there exists $y_1>0$ such that the above statement holds on the segment $\{1+iy : |y|<y_1\}$ with the same residue at the simple pole $s=1$. Furthermore, this $y_1$ doesn't depend on $H$ or $\rho$.
\end{proposition}
\begin{proof}
By reviewing equations \ref{eta} and \ref{spectral decomposition sum}, it is clear that we can write
$$\eta_{\rho}([H],s)=\sum_{k=1}^{\infty} \mathcal{L}_s^k ( \mathbb{1}_{[H]})$$
$$=\sum_{k=1}^{\infty}\left(\lambda_1(s)^k\mathcal{P}_{1,s}( \mathbb{1}_{[H]}) + ...+ \lambda_n(s)^k\mathcal{P}_{n,s}( \mathbb{1}_{[H]})+\mathcal{D}_{s}^k( \mathbb{1}_{[H]})\right).$$
Now we use proposition \ref{pressure decreasing} to see $|\lambda_i(s)|=e^{P(x)}<1$ if $x>1$. Therefore we can continue the above equation
$$=\lambda_1(s)(1-\lambda_1(s))^{-1}\mathcal{P}_{1,s}(\mathbb{1}_{[H]})+ ...+ \lambda_n(s)(1-\lambda_n(s))^{-1}\mathcal{P}_{n,s}(\mathbb{1}_{[H]})+\mathcal{Q}_s( \mathbb{1}_{[H]}),$$
where $\mathcal{Q}_s=\sum_{k=1}^{\infty}\mathcal{D}^k$ converges using lemma \ref{spectral D^m}. This is a valid relation for the Poincar\'{e} series $\eta_{\rho}$ on $x>1$. We fix $s_0$ on the line $x=1$, it is clear that $\mathcal{Q}_s(\mathbb{1}_{[H]})$ is a holomorphic function on the neighborhood $U$ of $s_0$ obtained in lemma \ref{spectral D^m}. Additionally all the projections $\mathcal{P}_{i,s}$ and function $\lambda_i(s)$ are analytic as discussed just above the equation \ref{spectral decomposition sum}. Therefore the right-hand side of the above equation is analytic on some neighborhood $U_0$ of $s_0$, as long as $\lambda_i(s_0)\neq1$. As we know for real $s=x+i0$, one of the eigenvalues of the transfer operator is $e^{P(x)}$ by theorem \ref{R. D. Mauldin and M. Urbanski}. We let $\lambda_1(s)$ represent this eigenvalue, it is clear that $\lambda_1(s)$ is not constant on any neighborhood of $s=1$ as $|\lambda_1(s)|=e^{P(x)}$ and $P$ is strictly decreasing by proposition \ref{pressure decreasing}. Since $\lambda_i(s)$ are isolated, simple eigenvalues and further analytic functions identity theorem from complex analysis guarantee the existence of $y_1>0$ for which the equations $\lambda_i(s)=1$ on $\{ 1+i y : |y|<y_1 \}$ have a solution only if $i=1$ and $s=1$. We deduce the righthand side of the equation above defines an analytic function on a neighborhood of $\{ 1+i y : 0<|y|<y_1 \}$. Note that $\lambda_1(s)$ is simple eigenvalue, so near $s=1$ we expect 
$$(1-\lambda_1(s)) \sim s-1 .$$
In other words, we find that $\eta_{\rho}([H],s)-A/(s-1)$ admits analytic extension to the segment $\{ 1+i y : |y|<y_1 \}$, where
$$A=\lim_{s \rightarrow 1} \eta_{\rho}([H],s)(s-1)=\lambda_1(1)\mathcal{P}_{1,1}(\mathbb{1}_{[H]})\lim_{s \rightarrow 1}\frac{s-1}{1-\lambda_1(s)}.$$
It is clear that using the D-generic property $y_1$ can be taken to be $\infty$. Thus, it only remains to compute $A$. It is clear that $\lambda_1=\lambda_1(1)=e^{P(1)}=1$. To compute $\mathcal{P}_{1,1}(\mathbb{1}_{[H]})$ first note that $\mathcal{L}_0\mathcal{P}_{i,1}=\lambda_i\mathcal{P}_{i,1}$ for each $i$, so $$\int \mathcal{P}_{i,1}(g)dm=\int \mathcal{L}_0\mathcal{P}_{i,1}(g)dm=\lambda_i \int \mathcal{P}_{i,1}(g)dm.$$ 
This gives $\int \mathcal{P}_{i,1}(g)=0$ for every $g \in C^{0,\alpha}(E_A^{\mathbb{N}},\mathbb{C})$ and $i\neq 1$. Therefore with respect to the measure $m$ for each $k$:
$$\int g= \int \mathcal{L}_0^k(g)=\int  \mathcal{P}_{1,1}(g)+ \int \lambda_2^k\mathcal{P}_{2,1}(g) + ...+ \int \lambda_n^k\mathcal{P}_{n,1}(g)+\int \mathcal{D}_1^k(g)$$
$$=\int  \mathcal{P}_{1,1}(g)+\int \mathcal{D}_{1}^k(g),$$
now implementing the inequality obtained in lemma \ref{spectral D^m} would yield 
$$\int g=\int  \mathcal{P}_{1,1}(g).$$
This actually determines the action of $\mathcal{P}_{1,1}$ since if $P_{1,1}(g)=k_g h$ then $k_g=\int g$, i.e. 
$$\mathcal{P}_{1,1}(g)=h \int gdm .$$
And lastly 
$$\lim_{s \rightarrow 1}\frac{1-\lambda_1(s)}{s-1}=\lim_{x \rightarrow 1}\frac{1-e^{P(x)}}{x-1}=-P'(1)e^{P(1)}=-\int f d\mu=\chi_{\mu},$$
where the equality to the last follows from proposition 2.6.13 in \citep[p. ~47]{MU}. Thus we find that the residue is $h(\rho)m([H])/\chi_{\mu}$.
\end{proof}

\section{Asymptotic Formula for Counting}\label{5}

In this section, we assume $f$ is strongly regular, summable and H\"{o}lder-type continuous with $P(1)=P(f)=0$. We keep this assumption to the end of a proposition \ref{borel non d-generic} and after that, we consider general functions with $P(\delta)=P(\delta f)=0$ for some $\delta>0$. We want to find an asymptotic formula for the counting functions presented in the previous section. We can provide a formula for some estimate of the lower bound and upper bound of all possible values. As well in this section by $y_0$ we mean 
$$y_0=\frac{y_1}{2\pi},$$ 
where $y_1$ was obtained in proposition \ref{residue}. As mentioned in that proposition, this $y_0$ doesn't depend on $H$ in $\eta_{\rho}([H],T)$. Further, we set
$$c_1:=y_0^{-1}\left(\exp(y_0^{-1})-1\right)^{-1}, \; \; c_2:=y_0^{-1}\left(\exp(y_0^{-1})-1\right)^{-1}\exp( y_0^{-1}).$$
\begin{proposition}\label{H proposition 2}
$$c_1 \dfrac{h(\rho)}{\chi_{\mu}}m([H])\leq \liminf_{T \rightarrow \infty}\dfrac{N_{\rho}(H, T)}{\exp( T)} \leq \limsup_{T \rightarrow \infty}\dfrac{N_{\rho}(H, T)}{\exp( T)} \leq c_2 \dfrac{h(\rho)}{\chi_{\mu}}m([H]),$$
and
$$c_1 \dfrac{h(\rho)}{\chi_{\mu}}m([H]) \leq \liminf_{T \rightarrow \infty}\dfrac{N_{\rho}([H], T)}{\exp( T)} \leq \limsup_{T \rightarrow \infty}\dfrac{N_{\rho}([H], T)}{\exp( T)} \leq c_2 \dfrac{h(\rho)}{\chi_{\mu}}m([H]),$$
and for every positive integer $q$
$$\lim_{T \rightarrow \infty}\dfrac{N_{\rho}([H],q, T)}{\exp( T)}=0.$$
\end{proposition}
\begin{proof}
The first two lines of inequalities follow from proposition \ref{residue} and applying Graham-Vaaler theorem \ref{Garahm-Vaaler}. The last equality follows from proposition \ref{eta holomorphic} and applying Ikehara-Wiener theorem \ref{ikehara-wiener delta}.
\end{proof}
\begin{proposition}\label{non D-generic tau}
$$c_1 \dfrac{1}{\chi_{\mu}}\mu([\tau])\leq \liminf_{T \rightarrow \infty}\dfrac{N_{\text{per}}(\tau, T)}{\exp( T)} \leq \limsup_{T \rightarrow \infty}\dfrac{N_{\text{per}}(\tau, T)}{\exp( T)}\leq c_2 \dfrac{1}{\chi_{\mu}}\mu([\tau]),$$
and 
$$c_1 \dfrac{1}{\chi_{\mu}}\mu([\tau])\leq \liminf_{T \rightarrow \infty}\dfrac{N_{\text{per}}([\tau], T)}{\exp( T)} \leq \limsup_{T \rightarrow \infty}\dfrac{N_{\text{per}}([\tau], T)}{\exp( T)} \leq c_2 \dfrac{1}{\chi_{\mu}}\mu([\tau]).$$
\end{proposition}
\begin{proof}
Let $\sum'$ represent the sum over all $\gamma$ with length $q$ such that $\tau\gamma$ is admissible. Then using part (ii) of lemma \ref{bound for periodic}, lemma \ref{liminf} and proposition \ref{H proposition 2} we can write:
\begin{align*}
\liminf_{T \rightarrow \infty} \dfrac{N_{\text{per}}(\tau, T)}{\exp( T)} & \geq \liminf_{T \rightarrow \infty} \sum ' \dfrac{N_{\tau\gamma(\tau\gamma)^+}(\tau\gamma, T-Ke^{-(k+q)\alpha})}{\exp( T)}\\
&\geq \exp\left(- Ke^{-(k+q)\alpha} \right)  \sum' \liminf_{T \rightarrow \infty}\dfrac{N_{\tau\gamma(\tau\gamma)^+}(\tau\gamma, T-Ke^{-(k+q)\alpha})}{\exp\left( T-Ke^{-(k+q)\alpha} \right)}\\
&=\exp\left(- Ke^{-(k+q)\alpha} \right)  \sum ' c_1 \dfrac{h(\tau\gamma(\tau\gamma)^+)}{\chi_{\mu}}m([\tau\gamma]).
\end{align*}
We use  lemma \ref{equivalent measure} at this step and continue:
\begin{align*}
\liminf_{T \rightarrow \infty} \dfrac{N_{\text{per}}(\tau, T)}{\exp( T)} & \geq c_1\frac{\exp\left(- Ke^{-(k+q)\alpha} \right)}{ \chi_{\mu}}  \sum' (1+K_1e^{-(k+q)\alpha})^{-1}\mu([\tau\gamma])\\
&=c_1 \frac{\exp\left(- Ke^{-(k+q)\alpha} \right)}{1+K_1e^{-(k+q)\alpha}}\frac{1}{ \chi_{\mu}}\mu([\tau]).
\end{align*}
Since $q$ is arbitrary, by $q \to \infty$ we obtain
$$\liminf_{T \rightarrow \infty} \dfrac{N_{\text{per}}(\tau, T)}{\exp( T)}\geq c_1\frac{1}{ \chi_{\mu}}\mu([\tau])$$
If we show
$$\limsup_{T \rightarrow \infty} \dfrac{N_{\text{per}}([\tau], T)}{\exp( T)}\leq c_2 \frac{1}{ \chi_{\mu}}\mu([\tau]),$$
we are done with the proof. We use lemma \ref{bound for periodic} part (v) for this and then we apply proposition \ref{H proposition 2} several times.
$$\limsup_{T \to \infty} \dfrac{N_{\text{per}}([\tau],T)}{\exp(T)}\leq \limsup_{T \to \infty} \sum\limits_{\substack{ \gamma \in F \\ \tau\gamma \in E_A^{\mathbb{N}}}} \dfrac{N_{\tau\gamma(\tau\gamma)^+}\left([\tau\gamma], T+Ke^{-(k+q)\alpha}\right)}{\exp(T)}$$
$$+\limsup_{T \to \infty}\sum\limits_{\substack{ \gamma \in F' \\ \tau\gamma E_A^{\mathbb{N}}}} \dfrac{N_{\tau\tau^+}([\tau\gamma], T+2K)}{\exp(T)}+\limsup_{T \to \infty} \sum_{i=1}^{k+q-1} \dfrac{N_{\tau\tau^+}([\tau],i,T+K)}{\exp(T)}.$$
Now the first limsup easily passes through the finite sum and we use proposition \ref{H proposition 2} with $H=\tau\gamma$, for the second limsup note that 
$$\sum\limits_{\substack{ \gamma \in F' \\ \tau\gamma E_A^{\mathbb{N}}}} N_{\tau\tau^+}([\tau\gamma], T+2K)=N_{\tau\tau^+}([\tau F'],T+2K),$$
therefore we apply proposition \ref{H proposition 2} with $H=\tau F'$ and the last limsup is clearly 0 using again proposition \ref{H proposition 2}. Thus we get
$$\limsup_{T \to \infty} \dfrac{N_{\text{per}}([\tau],T)}{\exp(T)}$$
$$\leq \sum\limits_{\substack{ \gamma \in F \\ \tau\gamma E_A^{\mathbb{N}}}} \limsup_{T \to \infty}\dfrac{N_{\tau\gamma(\tau\gamma)^+}\left([\tau\gamma], T+Ke^{-(k+q)\alpha}\right)}{\exp\left(T+Ke^{-(k+q)\alpha}\right)}\exp\left(Ke^{-(k+q)\alpha}\right)$$
$$+\limsup_{T \to \infty} \dfrac{N_{\tau\tau^+}([\tau F'],T+2K)}{\exp(T+2K)}\exp(2K)$$
$$= \exp\left(Ke^{-(k+q)\alpha}\right) \sum\limits_{\substack{ \gamma \in F \\ \tau\gamma E_A^{\mathbb{N}}}} c_2\dfrac{h(\tau\gamma(\tau\gamma)^+)}{\chi_{\mu}}m([\tau\gamma])+c_2\dfrac{h(\tau\tau^+)}{\chi_{\mu}}m([\tau F'])\exp(2K).$$
Notice that since $F$ was arbitrary for $\epsilon>0$ we choose $F$ such that $$c_2\dfrac{h(\tau\tau^+)}{\chi_{\mu}}m([\tau F'])\exp(2K) < \epsilon,$$
then we obtain
$$\limsup_{T \to \infty} \dfrac{N_{\text{per}}([\tau],T)}{\exp(T)}\leq \exp\left(Ke^{-(k+q)\alpha}\right) \sum\limits_{\substack{ \gamma \in F \\ \tau\gamma E_A^{\mathbb{N}}}} c_2\dfrac{h(\tau\gamma(\tau\gamma)^+)}{\chi_{\mu}}m([\tau\gamma])+\epsilon.$$
Now we apply  left-hand side of the lemma \ref{equivalent measure}:
$$\limsup_{T \to \infty} \dfrac{N_{\text{per}}([\tau],T)}{\exp(T)}\leq \dfrac{\exp\left(Ke^{-(k+q)\alpha}\right)}{1-K_1e^{-(k+q)\alpha}} \sum\limits_{\substack{ \gamma \in F \\ \tau\gamma E_A^{\mathbb{N}}}}c_2\dfrac{1}{\chi_{\mu}}\mu([\tau\gamma])+\epsilon.$$ 
Eventually we let $q \to \infty$ to get
$$\limsup_{T \to \infty} \dfrac{N_{\text{per}}([\tau],T)}{\exp(T)} \leq \sum\limits_{\substack{ \gamma \in F \\ \tau\gamma E_A^{\mathbb{N}}}}c_2\dfrac{1}{\chi_{\mu}}\mu([\tau\gamma]) +\epsilon = c_2 \dfrac{1}{\chi_{\mu}}\mu([\tau F])+\epsilon \leq c_2 \dfrac{1}{\chi_{\mu}}\mu([\tau]) +\epsilon.$$
Since $\epsilon$ was arbitrary we have
$$\limsup_{T \to \infty} \dfrac{N_{\text{per}}([\tau],T)}{\exp(T)} \leq c_2 \dfrac{1}{\chi_{\mu}}\mu([\tau]).$$
\end{proof}
\begin{proposition}\label{sup periodiccc}
$$\limsup_{T \rightarrow \infty}\dfrac{N_{\text{per}}(T)}{\exp( T)}\leq c_2\dfrac{1}{\chi_{\mu}}.$$
\end{proposition}
\begin{proof}
This proof is exactly similar to the proof of the previous proposition for limsup and implementing lemma \ref{bound for periodic whole}.
\end{proof}
\begin{proposition}\label{non d-generic open}
For every open set $V \subseteq E_A^{\mathbb{N}}$ we have
$$c_1\dfrac{h(\rho)}{\chi_{\mu}}m(V)\leq\liminf_{T \rightarrow \infty}\dfrac{N_{\rho}(V, T)}{\exp( T)}\leq\limsup_{T \rightarrow \infty}\dfrac{N_{\rho}(V, T)}{\exp( T)}\leq c_1 \dfrac{h(\rho)}{\chi_{\mu}}m(\overline{V})+y_0^{-1} \dfrac{h(\rho)}{\chi_{\mu}},$$
and
$$c_1\dfrac{1}{\chi_{\mu}}\mu(V) \leq \liminf_{T \rightarrow \infty}\dfrac{N_{\text{per}}(V, T)}{\exp( T)}\leq\limsup_{T \rightarrow \infty}\dfrac{N_{\text{per}}(V, T)}{\exp( T)} \leq c_1 \dfrac{1}{\chi_{\mu}}\mu(\overline{V})+y_0^{-1} \dfrac{1}{\chi_{\mu}}.$$
\end{proposition}
\begin{proof}
We know from proposition \ref{polish} that $V$ can be written as a union of disjoint cylinders, so $V=\cup_i [\tau_{(i)}]$. Therefore using lemma \ref{liminf} and proposition \ref{H proposition 2} with $H=\tau_{(i)}$ one can write
$$\liminf_{T \to \infty}\dfrac{N_{\rho}(V, T)}{\exp( T)}=\liminf_{T \to \infty}\sum_{i}\dfrac{N_{\rho}([\tau_{(i)}], T)}{\exp( T)}\geq\sum_{i} \liminf_{T \to \infty}\dfrac{N_{\rho}([\tau_{(i)}], T)}{\exp( T)}$$
$$\geq \sum_{i}c_1\dfrac{h(\rho)}{\chi_{\mu}}m([\tau_{(i)}])=c_1\dfrac{h(\rho)}{\chi_{\mu}}m(V).$$
For the limsup we use lemma \ref{f,g} and the above inequality for the open set $\overline{V}^{c}$ to find
$$c_1\dfrac{h(\rho)}{\chi_{\mu}}m(\overline{V}^{c})+\limsup_{T \to \infty}\dfrac{N_{\rho}(V, T)}{\exp( T)}$$
$$\leq \liminf_{T \to \infty}\dfrac{N_{\rho}(\overline{V}^{c}, T)}{\exp( T)}+\limsup_{T \to \infty}\dfrac{N_{\rho}(V, T)}{\exp( T)}$$
$$\leq \limsup_{T \to \infty}\dfrac{N_{\rho}(\overline{V}^{c}, T)+N_{\rho}(V, T)}{\exp( T)}\leq \limsup_{T \to \infty}\dfrac{N_{\rho}(T)}{\exp( T)} \leq c_2\dfrac{h(\rho)}{\chi_{\mu}},$$
where the last inequality holds if we apply proposition \ref{H proposition 2} for  $H=E$ (all the alphabets). This yields 
$$\limsup_{T \to \infty}\dfrac{N_{\rho}(V, T)}{\exp( T)} \leq c_2\dfrac{h(\rho)}{\chi_{\mu}}-c_1 \dfrac{h(\rho)}{\chi_{\mu}}m(\overline{V}^c)=c_1 \dfrac{h(\rho)}{\chi_{\mu}}m(\overline{V})+y_0^{-1} \dfrac{h(\rho)}{\chi_{\mu}}.$$
For counting periodic words, the idea is similar. Again we implement lemma \ref{liminf} and this time proposition \ref{non D-generic tau} to obtain:
$$\liminf_{T \to \infty}\dfrac{N_{\text{per}}(V, T)}{\exp( T)}=\liminf_{T \to \infty}\sum_{i}\dfrac{N_{\text{per}}([\tau_{(i)}], T)}{\exp( T)}\geq\sum_{i} \liminf_{T \to \infty}\dfrac{N_{\text{per}}([\tau_{(i)}], T)}{\exp( T)}$$
$$\geq \sum_{i}c_1\dfrac{1}{\chi_{\mu}}\mu([\tau_{(i)}])=c_1\dfrac{1}{\chi_{\mu}}\mu(V).$$
Applying lemma \ref{f,g} and the above inequality for the open set $\overline{V}^{c}$ gives us:
$$c_1\dfrac{1}{\chi_{\mu}}\mu(\overline{V}^{c})+\limsup_{T \to \infty}\dfrac{N_{\text{per}}(V, T)}{\exp( T)}$$
$$\leq \liminf_{T \to \infty}\dfrac{N_{\text{per}}(\overline{V}^{c}, T)}{\exp( T)}+\limsup_{T \to \infty}\dfrac{N_{\text{per}}(V, T)}{\exp( T)}$$
$$\leq \limsup_{T \to \infty}\dfrac{N_{\text{per}}(\overline{V}^{c}, T)+N_{\text{per}}(V, T)}{\exp( T)}\leq \limsup_{T \to \infty}\dfrac{N_{\text{per}}(T)}{\exp( T)} \leq c_2\dfrac{1}{\chi_{\mu}},$$
where the last inequality is due to the above proposition. This eventually gives 
$$\limsup_{T \to \infty}\dfrac{N_{\text{per}}(V, T)}{\exp( T)} \leq c_2\dfrac{1}{\chi_{\mu}}-c_1 \dfrac{1}{\chi_{\mu}}\mu(\overline{V}^c)=c_1 \dfrac{1}{\chi_{\mu}}\mu(\overline{V})+y_0^{-1} \dfrac{1}{\chi_{\mu}}.$$
\end{proof}
\begin{proposition}\label{borel non d-generic}
For every Borel set $B \subseteq E_A^{\mathbb{N}}$ we have
$$c_1\dfrac{h(\rho)}{\chi_{\mu}}m(B^o)\leq\liminf_{T \rightarrow \infty}\dfrac{N_{\rho}(B, T)}{\exp( T)}\leq\limsup_{T \rightarrow \infty}\dfrac{N_{\rho}(B, T)}{\exp( T)}\leq c_1 \dfrac{h(\rho)}{\chi_{\mu}}m(\overline{B})+y_0^{-1} \dfrac{h(\rho)}{\chi_{\mu}},$$
and
$$c_1 \dfrac{1}{\chi_{\mu}}\mu(B^o)\leq \liminf_{T \rightarrow \infty}\dfrac{N_{per}(B, T)}{\exp( T)}\leq\limsup_{T \rightarrow \infty}\dfrac{N_{per}(B, T)}{\exp( T)}\leq c_1 \dfrac{1}{\chi_{\mu}}\mu(\overline{B})+y_0^{-1} \dfrac{1}{\chi_{\mu}}.$$
\end{proposition}
\begin{proof}
We only prove the first line of inequalities. The other one is proved in a similar manner. We apply the above proposition to  open set $B^o$:
$$c_1\dfrac{h(\rho)}{\chi_{\mu}}m(B^o)\leq \liminf_{T \rightarrow \infty}\dfrac{N_{\rho}(B^o, T)}{\exp( T)}\leq \liminf_{T \rightarrow \infty}\dfrac{N_{\rho}(B, T)}{\exp( T)}.$$
For limsup we use lemma \ref{f,g} and this inequality for $\overline{B}^{c}$:
$$c_1\dfrac{h(\rho)}{\chi_{\mu}}m(\overline{B}^{c})+ \limsup_{T \rightarrow \infty}\dfrac{N_{\rho}(\overline{B}, T)}{\exp( T)}$$
$$\leq \liminf_{T \rightarrow \infty}\dfrac{N_{\rho}(\overline{B}^{c}, T)}{\exp( T)}+\limsup_{T \rightarrow \infty}\dfrac{N_{\rho}(\overline{B}, T)}{\exp( T)}\leq \limsup_{T \rightarrow \infty}\dfrac{N_{\rho}( T)}{\exp( T)}\leq c_2\dfrac{h(\rho)}{\chi_{\mu}},$$
where the last inequality holds if we apply proposition \ref{H proposition 2} for  $H=E$ (all the alphabets). Thus
$$\limsup_{T \rightarrow \infty}\dfrac{N_{\rho}(B, T)}{\exp( T)} \leq \limsup_{T \rightarrow \infty}\dfrac{N_{\rho}(\overline{B}, T)}{\exp( T)} \leq c_2\dfrac{h(\rho)}{\chi_{\mu}}- c_1\dfrac{h(\rho)}{\chi_{\mu}}m(\overline{B}^c)=c_1 \dfrac{h(\rho)}{\chi_{\mu}}m(\overline{B})+y_0^{-1} \dfrac{h(\rho)}{\chi_{\mu}}.$$
This finishes the proof.
\end{proof}
Note that so far we focused on the systems with $P(1)=0.$ We want to show that this is not restrictive and we can otherwise get the corresponding counting formula as well. For a general H\"{o}lder-type function $f: E_A^{\mathbb{N}} \to \mathbb{R}$ we remember that $x \in \Gamma$ iff $xf$ is summable. Assuming strong regularity we know there exists $\delta>0$ such that $P(\delta)=0$ and $\inf \Gamma <\delta$. Now if we consider a new function $g=\delta f$, first it is clear that $g$ as well is strongly regular. Secondly, since $P(xg)=P(x\delta f)$ we have $P_{g}(1)=0$. Therefore all the results obtained above are applicable for $g$. Additionally, note that 
$S_ng(\rho)=\delta S_nf(\rho),$
so we find that
\begin{equation}\label{N_g N_f}
N^g(\delta T)=N(T).
\end{equation}
Moreover, it is clear 
$$\mathcal{L}_{1g}=\mathcal{L}_{\delta f}.$$
Therefore if $\mathcal{L}_{\delta f}h_{\delta}=h_{\delta}$ then $\mathcal{L}_{1g}h_{\delta}=h_{\delta}$, similarly if $\mathcal{L}_{\delta f}^*m_{\delta}=m_{\delta}$ then $\mathcal{L}_{1 g}^*m_{\delta}=m_{\delta}$. Additionally, if $\mathcal{L}_{s f}$ avoids  $\exp\left(P(\delta f)\right)=1$ as eigenvalue on
$$\{\delta+iy : 0<|y|<y_0(f)\},$$
then $\mathcal{L}_{sg}$ does so on
$$\{1+iy : 0<|y|<\frac{y_0(f)}{\delta}\}.$$
This implies $y_0(g)=\frac{y_0(f)}{\delta}$ and so 
$$c_1(g)=(\frac{y_0}{\delta})^{-1}\left(\exp((\frac{y_0}{\delta})^{-1})-1\right)^{-1}.$$
It is now enough to use proposition \ref{borel non d-generic} for $g$ with $m=m_{\delta}$, $\mu=\mu_{\delta}$ and $h=h_{\delta}$ to estimate:
$$c_1(g) \dfrac{h(\rho)}{\chi_{\mu}}m(B^o)\leq\liminf_{T \rightarrow \infty}\dfrac{N_{\rho}^g(B, T)}{\exp( T)}\leq\limsup_{T \rightarrow \infty}\dfrac{N_{\rho}^g(B, T)}{\exp( T)}\leq c_1(g) \dfrac{h(\rho)}{ \chi_{\mu}}m(\overline{B})+(\frac{y_0}{\delta})^{-1} \dfrac{h(\rho)}{ \chi_{\mu}}.$$
Furthermore, note that
$$\chi_{\mu}=-\int g \text{ d}\mu_{\delta}=-\delta \int f \text{ d}\mu_{\delta}=\delta \chi_{\mu_{\delta}}.$$
Now we replace $T$ with $\delta T$ and use \ref{N_g N_f} to obtain the following estimate for $f$:
$$y_0^{-1}\left(\exp(\delta y_0^{-1})-1\right)^{-1} \dfrac{h_{\delta}(\rho)}{ \chi_{\mu_{\delta}}}m_{\delta}(B^o)\leq\liminf_{T \rightarrow \infty}\dfrac{N_{\rho}(B, T)}{\exp(\delta T)}$$
$$\leq\limsup_{T \rightarrow \infty}\dfrac{N_{\rho}(B, T)}{\exp(\delta T)}\leq y_0^{-1}\left(\exp(\delta y_0^{-1})-1\right)^{-1} \dfrac{h_{\delta}(\rho)}{ \chi_{\mu_{\delta}}}m_{\delta}(\overline{B})+y_0^{-1} \dfrac{h_{\delta}(\rho)}{ \chi_{\mu_{\delta}}}.$$
Similarly, we can obtain a formula for $N_{\text{per}}(B,T)$ which we omit its proof. We set
\begin{equation}\label{c-delta}
c_{\delta}:=y_0^{-1}\left(\exp(\delta y_0^{-1})-1\right)^{-1}   
\end{equation}
and capture all the aforementioned arguments in the following theorem.
\begin{theorem}\label{main thm}
If $f: E_A^{\mathbb{N}} \to \mathbb{R}$ is strongly regular H\"{o}lder-type function with $P(\delta f)=0$, for every Borel set $B \subseteq E_A^{\mathbb{N}}$ and $\rho \in E_A^{\mathbb{N}}$ we have
$$c_{\delta}\dfrac{h_{\delta}(\rho)}{ \chi_{\mu_{\delta}}}m_{\delta}(B^o)\leq\liminf_{T \rightarrow \infty}\dfrac{N_{\rho}(B, T)}{\exp(\delta T)}\leq\limsup_{T \rightarrow \infty}\dfrac{N_{\rho}(B, T)}{\exp( \delta T)}\leq c_{\delta} \dfrac{h_{\delta}(\rho)}{ \chi_{\mu_{\delta}}}m_{\delta}(\overline{B})+y_0^{-1} \dfrac{h_{\delta}(\rho)}{ \chi_{\mu_{\delta}}},$$
and
$$c_{\delta}\frac{1}{\ \chi_{\mu_{\delta}}}\mu(B^o)\leq \liminf_{T \rightarrow \infty}\dfrac{N_{per}(B, T)}{\exp( \delta T)}\leq\limsup_{T \rightarrow \infty}\dfrac{N_{per}(B, T)}{\exp( \delta T)}\leq c_{\delta} \dfrac{1}{ \chi_{\mu_{\delta}}}\mu(\overline{B})+y_0^{-1} \dfrac{1}{ \chi_{\mu_{\delta}}}.$$
\end{theorem}
\begin{remark}
It is important to note that
\begin{itemize}
    \item For $N_{\rho}$ the eigenmeasure $m$ and for $N_{per}$ the equilibrium measure $\mu$ appears in the formula.
    \item The bounds are sharp as shown in example \ref{sharpness} below.
    \item The limit points of the ratio $\frac{N_{\rho}(B,T)}{\exp(\delta T)}$ can be a full closed interval, i.e.
    $$\left\{A \; : \; A=\lim_{n \to \infty}\frac{N_{\rho}(B,T_n)}{\exp(\delta T_n)}, \; \; T_n \to \infty \; \; \text{as}\; \; n \to \infty \right\}=[c,C],$$
    for some $c,C>0$. (see example \ref{sharpness})
\end{itemize}

\end{remark}
\begin{corollary}\label{border 0 non-degeneric}
If $f: E_A^{\mathbb{N}} \to \mathbb{R}$ is strongly regular H\"{o}lder-type function with $P(\delta f)=0$, for every Borel set $B \subseteq E_A^{\mathbb{N}}$ with boundary of measure $0$ and $\rho \in E_A^{\mathbb{N}}$ we have
$$c_{\delta}\dfrac{h_{\delta}(\rho)}{\chi_{\mu}}m(B)\leq \liminf_{T \rightarrow \infty}\dfrac{N_{\rho}(B,  T)}{\exp(\delta T)}\leq \limsup_{T \rightarrow \infty}\dfrac{N_{\rho}(B, T)}{\exp(\delta T)}\leq c_{\delta} \dfrac{h_{\delta}(\rho)}{\chi_{\mu}}m(B)+y_0^{-1} \dfrac{h_{\delta}(\rho)}{\chi_{\mu}},$$
and 
$$c_{\delta}\dfrac{1}{\chi_{\mu}}\mu(B)\leq \liminf_{T \rightarrow \infty}\dfrac{N_{\text{per}}(B, T)}{\exp(\delta T)}\leq\limsup_{T \rightarrow \infty}\dfrac{N_{\text{per}}(B, T)}{\exp(\delta T)}\leq c_{\delta} \dfrac{1}{\chi_{\mu}}\mu(B)+y_0^{-1} \dfrac{1}{\chi_{\mu}}.$$
\end{corollary}
\begin{proof}
We just need to apply the above theorem and note that $m(\partial B)=0$ implies $m(B)=m(\overline{B})=m(B^o).$
\end{proof}
\begin{corollary}[Pollicott-Urba\'{n}ski]\label{llll}
Let $\mathcal{S}=\{ \phi_e \}_{e \in E}$ be a strongly regular conformal graph directed Markov system with D-generic property. Let $\delta$ be the Hausdorff dimension of the limit set of $\mathcal{S}$, then for every Borel set $B \subseteq E_A^{\mathbb{N}}$ with boundary of measure $0$ and $\rho \in E_A^{\mathbb{N}}$ we have
$$\lim_{T \rightarrow \infty}\dfrac{N_{\rho}(B, T)}{\exp( \delta T)}=\dfrac{h_{\delta}(\rho)}{\delta \chi_{\mu_{\delta}}}m_{\delta}(B),$$
and 
$$\lim_{T \rightarrow \infty}\dfrac{N_{\text{per}}(B, T)}{\exp(\delta T)}=\dfrac{1}{\delta \chi_{\mu_{\delta}}}\mu_{\delta}(B).$$
\end{corollary}
\begin{proof}
It follows from the previous corollary. Note that when $\mathcal{S}$ is D-generic then we are allowed to let $y_0 \to \infty$ and this gives $c_{\delta} \to \frac{1}{\delta}$ from \ref{c-delta}.
\end{proof}
\section{Asymptotic Formula for Length}
Before bringing some examples we would like to talk about counting with specified length. As indicated at the beginning of the previous section item (d) we had $N_{\rho}([H],q,T)$ which is counting the number of words $\omega$ satisfying $S_{|\omega|}f(\omega \rho)\geq -T$ of length $q$. We addressed in proposition \ref{H proposition 2} that growth of this relative to $\exp(\delta T)$ tends to $0$. Therefore if we would like to obtain fairly interesting growth we have to focus on some counting where $q$ as well grows as $T$ grows. We know $N_{\rho}(T) \sim C \exp(\delta T)$ but if we write 
$$N_{\rho}(T)=\sum_{i=1}^{\infty} N_{\rho}(i,T),$$
first, we should note that this sum is terminating at some point. More precisely, for $\rho$ if we set
$$m(T):= \sup_{\omega \in E_{\rho}^*} \{ |\omega|  :  S_{|\omega'|}f(\omega'\rho)\geq -T, \; \forall \omega' \in E_{\rho}^*, \; |\omega'|\leq |\omega| \}, \hspace{1cm} b_n:=\inf_{\omega \in E_{\rho}^n}  S_n(\omega\rho),$$
$$ M(T):=\sup_{\omega \in E_{\rho}^*} \{ |\omega| \; : \; S_{|\omega|}f(\omega\rho)\geq -T \}, \hspace{4.5cm} d_n:=\sup_{\omega \in E_{\rho}^n}  S_n(\omega\rho),$$
then $N_{\rho}(i,T)=0$ for $i>M(T)$, therefore
$$N_{\rho}(T)=\sum_{i=1}^{M(T)} N_{\rho}(i,T).$$
The question we ask is which term of the above sum on the right-hand side might have growth comparable to the left-hand side, i.e. for which $i(T)$ the growth of $N_{\rho}(T)/N_{\rho}(i(T),T)$ is not too fast?! With the tools we have, we couldn't answer this question, however, we have some words on that. First, we prove the following.
\begin{proposition}\label{asymp length}
Both of the following limits exist:
$$\lim_{T \to \infty}\frac{m(T)}{T}=r, \; \; \; \; \lim_{T \to \infty}\frac{M(T)}{T}=s.$$
\end{proposition}
\begin{proof}
First, we prove the latter. We set $M:=M(T)$, let $\omega$ be a finite word making the supremum possible in the definition of $M(T)$, then for any $\tau \in E_{\rho}^{M+1}$ we find
$$d_M\geq S_Mf(\omega \rho)\geq -T > S_{M+1}(\tau\rho),$$
$$d_M\geq -T \geq d_{M+1},$$
$$\frac{d_M}{M}\geq \frac{-T}{M} \geq \frac{d_{M+1}}{M+1}\frac{M+1}{M}.$$
Therefore it is enough to show that $d_n/n$ is convergent. To do so, we note that for arbitrary $\tau$, $\gamma$ with $|\tau|=m$, $|\gamma|=n$ where $\tau\gamma\rho$ is admissible, we can find $\omega \in \Omega$ such that $\tau\omega\rho$ is as well admissible by finitely irreducible definition \ref{finitely irreducible}. By lemma \ref{bounded S_n} we find:
$$\delta S_{m+n}f(\tau\gamma\rho)=\delta S_mf(\tau\gamma\rho)+\delta S_nf(\gamma\rho)\leq \delta S_mf(\tau\omega\rho)+\delta S_nf(\gamma\rho)+K_{\delta f}$$ 
$$=\delta S_{m+|\omega|}f(\tau\omega\rho)-\delta S_{|\omega|}f(\omega\rho)+\delta S_nf(\gamma\rho)+K_{\delta f}$$
$$=\delta S_{|\omega|}f(\tau\omega\rho)+\delta S_{m}f(\sigma^{|\omega|}(\tau\omega\rho))-\delta S_{|\omega|}f(\omega\rho)+\delta S_nf(\gamma\rho)+K_{\delta f}$$
Now by \ref{Gibbs} we know $\delta S_{|\omega|}f \leq \log Q_{\delta}$ and since $\Omega$ is finite, there is $C>0$ such that
$$S_{m+n}f(\tau\gamma\rho)\leq S_{m}f(\sigma^{|\omega|}(\tau\omega\rho))+S_nf(\gamma\rho)+C\leq d_m+d_n+C.$$
Thus we have $d_{m+n}\leq d_m+d_n+C$ and we can use Fekete's lemma \ref{fekete} with $a_n=d_n+C$ to get convergence of $d_n/n$.\\
For the other one, note that if $E$ is infinite then using \ref{Gibbs} there are infinitely many $n$ for which $b_n=-\infty$, therefore $m(T)=\sup \emptyset$ which we set it $-\infty$ and so $m(T)/T=-\infty$ for all $T>0$. Let $E$ be finite, for arbitrary $\tau$, $\gamma$ with $|\tau|=m$, $|\gamma|=n$ where $\tau\rho$ and $\gamma\rho$ are admissible there is $\omega \in \Omega$ such that $\tau\omega\gamma\rho$ is admissible as well. Therefore by lemma \ref{bounded S_n} we find:
$$\delta S_mf(\tau\rho)+\delta S_nf(\gamma\rho)\geq \delta S_mf(\tau\omega\gamma\rho)-K_{\delta f}+\delta S_nf(\gamma\rho)$$
$$=\delta S_{m+|\omega|}f(\tau\omega\gamma\rho)-\delta S_{|\omega|}f(\omega\gamma\rho)-K_{\delta f}+\delta S_nf(\gamma\rho)$$
$$=\delta S_{m+|\omega|+n}f(\tau\omega\gamma\rho)-\delta S_{|\omega|}f(\omega\gamma\rho)-K_{\delta f}$$
$$=\delta S_{|\omega|}f(\tau\omega\gamma\rho)+\delta S_{m+n}f(\sigma^{|\omega|}(\tau\omega\gamma\rho))-\delta S_{|\omega|}f(\omega\gamma\rho)-K_{\delta f}.$$
Now for large $m$ it is clear that by lemma \ref{bounded S_n} we have $\delta S_{|\omega|}f(\tau\omega\gamma\rho)\geq \delta S_{|\omega|}f(\tau\rho)-K_{\delta f}$, so again we use \ref{Gibbs} and the fact that $E$ and $\Omega$ are finite to obtain $C>0$ such that:
$$\delta S_mf(\tau\rho)+\delta S_nf(\gamma\rho)\geq \delta S_{m+n}f(\sigma^{|\omega|}(\tau\omega\gamma\rho))-C.$$
This gives $b_m+b_n\geq b_{m+n}-C$, and once again we use Fekete's lemma to find that $b_n/n$ is convergent. Note that similar to above we can set $m:=m(T)$ and let $\omega$ be a finite word making the supremum possible in the definition of $m(T)$, so:
$$\frac{b_m}{m}\geq \frac{-T}{m} \geq \frac{b_{m+1}}{m+1}\frac{m+1}{m}.$$
This finishes the proof.
\end{proof}
Note that $m(T)$ is the cutoff integer where before that the counting problem is just counting $\sum_{i=1}^{m(T)} \# E_A^i$, while after that not all words with generic length are included in $N_{\rho}(T)$. We continue this omitting process till we reach $M(T)$ where no finite word of length bigger is counted anymore. Furthermore, it is obvious that $r\leq s$. We know equality and strict inequality are both possible, examples \ref{sharpness}, \ref{deterministic 1/2, 1/3} correspondingly. Our guess is the following
$$\frac{N_{\rho}(T)}{N_{\rho}(i(T),T)}=O(T) \iff \frac{i(T)}{T}\to \frac{1}{\chi_{\mu_{\delta}}},$$
where $O$ is just the big O notation and $m(T)\leq i(T)\leq M(T)$. As stated, we couldn't show this with the tools we have. Note that this last assumption cannot be relaxed, for taking $i(T)=M(T)+1$ in example \ref{sharpness} gives 
$$\frac{N_{\rho}(T)}{1+N_{\rho}(i(T),T)}=N_{\rho}(T)=O\left(\exp(\delta T)\right), \; \; \frac{i(T)}{T} \to \frac{1}{\chi_{\mu_{\delta}}}.$$
In example \ref{sharpness} we have only one choice $i(T)=m(T)=M(T)$ and then $N_{\rho}(T)/N_{\rho}(i(T),T)=1$. However, computations get much harder for example \ref{deterministic 1/2, 1/3}. Our computations using an asymptotic formula for partial sum of binomials \citep[~ p. 492]{CS} suggest $N_{\rho}(T)/N_{\rho}(i(T),T)=O(T)$. In case, such a relation holds in general, it tells us that the main contributor to $N_{\rho}(T)$ is asymptotically $N_{\rho}(i(T),T)$. This is important because in some cases one needs to deal with words of specified length rather than any length when working with $N_{\rho}(T)$.
\section{Examples}
\begin{example}
Recalling example \ref{alpha_i constant} from the preliminaries section, we apply theorem \ref{main thm}. We know that the transfer operator $\mathcal{L}_s$ for real $s=x$ due to Ruelle's theorem \citep[~ p. 136]{P-RPF}, has only one eigenvalue of modulus $e^{P(x)}$ and this eigenvalue is $e^{P(x)}$. Therefore by equation \ref{Pp} for real $s=x$ the eigenvalue is of the form
$$\lambda(x)=\exp\left(\log r(A) +x\log \alpha\right).$$
and since eigenvalue is an analytic function using the identity theorem we obtain for every $s$
$$\lambda(s)=\exp\left(\log r(A) +s\log \alpha\right).$$
We observe that $\lambda(s)=1$ when 
$$s=\frac{ \log r(A)}{-\log \alpha}+\frac{2 \pi  k}{  -\log \alpha}i,  \; \; \; k \in \mathbb{Z}.$$ 
Therefore $\delta=\log r(A)/-\log \alpha$ and  $\eta_{\rho}-1/(s-\delta)$ has continuous extension on the segment 
$$\{s \in \mathbb{C} : s=\delta+iy,\; \;  |y|< \frac{2 \pi}{-\log \alpha } \}$$
of the critical line. Then  theorem \ref{main thm} for 
$$y_0=\frac{y_1}{2 \pi}=\frac{-1}{\log \alpha},\; \; \delta=\frac{\log r(A)}{-\log \alpha}, \; \; \chi_{\mu_{\delta}}=-\int \log \alpha \text{d}\mu_{\delta}=-\log \alpha$$
gives us the following estimate:
\begin{equation}\label{eqr 7}
\frac{h_{\delta}(\rho)}{r(A)-1} \leq \liminf_T \dfrac{N_{\rho} (T)}{\exp(\delta T)}\leq \limsup_T \dfrac{N_{\rho} (T)}{\exp(\delta T)} \leq \frac{h_{\delta}(\rho)}{r(A)-1}r(A).
\end{equation}
\end{example}
\begin{example}\label{sharpness}
In the previous example if we consider the full shift with two letters and $\alpha=\frac{1}{3}$ with maps
$$\phi_0(x)=\frac{1}{3}x, \; \; \; \phi_1(x)=\frac{1}{3}x+\frac{2}{3},$$
then the limit set of this system is the Cantor set on the unit interval. Therefore
$$f(\rho)=\log |\phi_{\rho_1}'(\pi(\sigma\rho))|=\log \frac{1}{3},$$
$$\mathcal{L}_{s} \mathbb{1}(\rho)=\exp(s f(0 \rho))+\exp(sf(1 \rho)) =2(\frac{1}{3^s}), \; \; h=\mathbb{1}$$
$$\log r(A)=\lim_n \frac{1}{n}\log \# E^n=\lim_n \frac{1}{n}\log 2^n=\log 2.$$
$$  y_0=\frac{1}{\log 3}, \; \; \delta=\frac{\log 2}{\log 3}$$
Thus
\begin{equation}\label{eqr 0}
1 \leq \liminf_T \dfrac{N_{\rho} (T)}{\exp(\delta T)}\leq \limsup_T \dfrac{N_{\rho} (T)}{\exp(\delta T)} \leq  2.
\end{equation}
Now we show that the left and right inequalities in the above line are actually equalities. For this, we need to explicitly compute $N_{\rho}(T)$. If $|\omega|=n$:
$$S_n f(\omega \rho)= n\log \frac{1}{3}$$
$$  S_nf(\omega \rho)\geq -T \iff n\leq \frac{T}{\log 3}. $$
Therefore
$$\frac{N(T)}{\exp(\delta T)}=\frac{2^{\lfloor \frac{T}{\log 3} \rfloor +1}-2}{2^{\frac{T}{\log 3}}}.$$
We dropped the notation $\rho$ in $N_{\rho}(T)$ as it is independent. For any $A \in [0,1]$ it is clear that we can choose a sequence $T_n$ with $T_n \to \infty$ such that $\frac{T_n}{\log 3}-{\lfloor  \frac{T_n}{\log 3}\rfloor} \to A$. Then we obtain:
$$ \dfrac{N (T_n)}{\exp(\delta T_n)} \to 2^{-A+1},
$$
This in particular means
$$\liminf_{T \infty} \frac{N (T)}{\exp(\delta T)}=1, \; \; \; \limsup_{ T \to \infty} \dfrac{N (T)}{\exp(\delta T)}=2.$$
\end{example}
Note that in general computing $N_{\rho}(T)$ is not so easy even for simple systems. The following example is one in that regard.
\begin{example}\label{deterministic 1/2, 1/3}
Recalling example \ref{example alpha different}, consider the deterministic system with conformal maps of the unit interval 
$$\phi_0(x)=\frac{1}{2}x+\frac{1}{20}, \; \; \; \phi_1(x)=\frac{1}{3}x+\frac{1}{30}$$
on the full shift space $E^{\infty}=\{0,1\}^{\infty}$. Clearly, we have
$$f(\rho)=\log |\phi_{\rho_1}'(\sigma \rho)|,$$
$$S_n f(\omega \rho)=n_0 \log \frac{1}{2}+n_1 \log \frac{1}{3},$$ 
where $n_0=n_0(\omega)=S_n\mathbbm{1}_{[0]}(\omega\rho)$ and $n_1=n_1(\omega)=S_n\mathbbm{1}_{[1]}(\omega\rho)$. Basically, $n_0$ is the number of $0$s  and $n_1$ is the number of $1$s in $\omega \in E^n$. The pressure is calculated to be
\begin{align*}
P(x)=\lim_n \frac{1}{n}\log\sum_{|\omega|=n}\|\phi'_{\omega}\|^x&=\lim_n \frac{1}{n}\log\sum_{|\omega|=n}(\frac{1}{2^{n_0(\omega)}}\frac{1}{3^{n_1(\omega)}})^x\\
                                                                                                         &=\lim_n \frac{1}{n} \log (\frac{1}{2^x}+\frac{1}{3^x})^n=\log (\frac{1}{2^x}+\frac{1}{3^x})
\end{align*}
And a Gibbs state by $\ref{eqr 10}$ can be found first on $[\omega_1]$, then on $[\omega_1\omega_2]$ and so on:
$$m_x\left([\omega]\right)=\dfrac{(\frac{1}{2^x})^{n_0}(\frac{1}{3^x})^{n_1}}{(\frac{1}{2^x}+\frac{1}{3^x})^n}, \; \; \; \omega \in E^n.$$
Note that, it defines a system with D-generic property. One way to see that the system is D-generic is by Proposition \ref{prop. d-generic}.
Note that $E_{\text{per}}^*$ is the set of periodic words of any length which is exactly $E^*$, since we work with the full shift. Therefore if the set 
$$\{S_{|\omega|}f(\overline{\omega}) : \omega \in E_{\text{per}}^*\}=\{n_0 \log \frac{1}{2}+n_1 \log \frac{1}{3} : n_0+n_1=n \in \mathbb{N} \},$$
generates a cyclic additive group with a generator $\beta$, then there exist integers $k,k'$ such that $k\beta=\log 1/2$ and  $k'\beta=\log 1/3$. This yields $k/k'=\log 2/\log 3$ is rational. The other way to see that our system has D-generic property is by directly solving the following equation for the eigenvalue of the maximal modulus of the transfer operator:
$$1=\lambda(s)=\frac{1}{2^s}+\frac{1}{3^s}, \; \;  x=\delta$$
$$|\frac{1}{2^s}+\frac{1}{3^s}|=1=\frac{1}{2^{\delta}}+\frac{1}{3^{\delta}}=|\frac{1}{2^s}|+|\frac{1}{3^s}|,$$
so by properties of the triangle inequality, there exists $b\geq 0$ such that
$$\frac{1}{2^s}=b\frac{1}{3^s} \; \Rightarrow \; b=\frac{3^s}{2^s}=\frac{3^{\delta}}{2^{\delta}}\exp(iy\log 3 -iy\log 2) \Rightarrow \; b=\frac{3^{\delta}}{2^{\delta}}, \; \;  y=\frac{2k\pi}{\log(3/2)}.$$
But,
$$1=\frac{1}{2^s}+\frac{1}{3^s}=b\frac{1}{3^s}+\frac{1}{3^s}=b\frac{1}{3^{\delta}}\exp(-iy\log 3)+\frac{1}{3^{\delta}}\exp(-iy\log 3)$$
$$=\frac{1}{2^{\delta}}\exp(-iy\log 3)+\frac{1}{3^{\delta}}\exp(-iy\log 3)=\exp(-iy\log 3) \Rightarrow \; y=\frac{2k\pi}{\log3},$$
i.e. $y$ can only be $0$.
Now we are ready to apply corollary \ref{llll} to find 
$$\frac{N_{\rho}(T)}{\exp(\delta T)} \rightarrow \frac{1}{\delta \chi_{\mu_{\delta}} }, \; \; \; T \to \infty . $$ 
\end{example}
In order to keep this paper short we didn't bring examples with infinite alphabets. One example would be the Apollonian circle packing problem for which there is a good exposition already in \citep[~ ch. 6]{UP}.
\end{large}

\bibliographystyle{alpha}
\bibliography{bibbb.bib}

\begin{thebibliography}{URM21}

\bibitem[Bau85]{Baumg}
H.~Baumg\"{a}rtel.
\newblock {\em Analytic perturbation theory for matrices and operators},
  volume~15 of {\em Operator Theory: Advances and Applications}.
\newblock Birkh\"{a}user Verlag, Basel, 1985.

\bibitem[Boy73]{Boyd1}
David~W. Boyd.
\newblock The residual set dimension of the {A}pollonian packing.
\newblock {\em Mathematika}, 20:170--174, 1973.

\bibitem[Boy82]{Boyd2}
David~W. Boyd.
\newblock The sequence of radii of the {A}pollonian packing.
\newblock {\em Math. Comp.}, 39(159):249--254, 1982.

\bibitem[Bro61]{Browder}
Felix~E. Browder.
\newblock On the spectral theory of elliptic differential operators. {I}.
\newblock {\em Math. Ann.}, 142:22--130, 1960/61.

\bibitem[Con90]{Conway}
John~B. Conway.
\newblock {\em A course in functional analysis}, volume~96 of {\em Graduate
  Texts in Mathematics}.
\newblock Springer-Verlag, New York, second edition, 1990.

\bibitem[Del42]{Delsarte}
Jean Delsarte.
\newblock Sur le gitter fuchsien.
\newblock {\em C. R. Acad. Sci. Paris}, 214:147--179, 1942.

\bibitem[DS58]{Dunford}
Nelson Dunford and Jacob~T. Schwartz.
\newblock {\em Linear {O}perators. {I}. {G}eneral {T}heory}.
\newblock Interscience Publishers, Inc., New York; Interscience Publishers,
  Ltd., London, 1958.

\bibitem[Fol99]{Folland}
Gerald~B. Folland.
\newblock {\em Real analysis}.
\newblock Pure and Applied Mathematics (New York). John Wiley \& Sons, Inc.,
  New York, second edition, 1999.
\newblock Modern techniques and their applications, A Wiley-Interscience
  Publication.

\bibitem[GKP94]{CS}
Ronald~L. Graham, Donald~E. Knuth, and Oren Patashnik.
\newblock {\em Concrete mathematics}.
\newblock Addison-Wesley Publishing Company, Reading, MA, second edition, 1994.
\newblock A foundation for computer science.

\bibitem[GT01]{Holder}
David Gilbarg and Neil~S. Trudinger.
\newblock {\em Elliptic partial differential equations of second order}.
\newblock Classics in Mathematics. Springer-Verlag, Berlin, 2001.
\newblock Reprint of the 1998 edition.

\bibitem[GV81]{Vaaler}
S.~W. Graham and Jeffrey~D. Vaaler.
\newblock A class of extremal functions for the {F}ourier transform.
\newblock {\em Trans. Amer. Math. Soc.}, 265, 1981.

\bibitem[Haf81]{Hafner}
James~Lee Hafner.
\newblock New omega theorems for two classical lattice point problems.
\newblock {\em Invent. Math.}, 63(2):181--186, 1981.

\bibitem[Hub56]{HUBER1959}
Heinz Huber.
\newblock \"{U}ber eine neue {K}lasse automorpher {F}unktionen und ein
  {G}itterpunktproblem in der hyperbolischen {E}bene. {I}.
\newblock {\em Comment. Math. Helv.}, 30:20--62 (1955), 1956.

\bibitem[Hux90]{Huxley}
M.~N. Huxley.
\newblock Exponential sums and lattice points.
\newblock {\em Proc. London Math. Soc. (3)}, 60(3):471--502, 1990.

\bibitem[IM88]{Iwaniec}
H.~Iwaniec and C.~J. Mozzochi.
\newblock On the divisor and circle problems.
\newblock {\em J. Number Theory}, 29(1):60--93, 1988.

\bibitem[Ivi85]{Ivic}
Aleksandar Ivi\'{c}.
\newblock {\em The {R}iemann zeta-function}.
\newblock A Wiley-Interscience Publication. John Wiley \& Sons, Inc., New York,
  1985.
\newblock The theory of the Riemann zeta-function with applications.

\bibitem[Kat52]{Kato2}
Tosio Kato.
\newblock On the perturbation theory of closed linear operators.
\newblock {\em J. Math. Soc. Japan}, 4:323--337, 1952.

\bibitem[Kat76]{TK}
Tosio Kato.
\newblock {\em Perturbation theory for linear operators}.
\newblock Springer-Verlag, Berlin-New York, second edition, 1976.
\newblock Grundlehren der Mathematischen Wissenschaften, Band 132.

\bibitem[KO11]{Kontorovich}
Alex Kontorovich and Hee Oh.
\newblock Apollonian circle packings and closed horospheres on hyperbolic
  3-manifolds.
\newblock {\em J. Amer. Math. Soc.}, 24(3):603--648, 2011.
\newblock With an appendix by Oh and Nimish Shah.

\bibitem[Kor04]{Korevaar}
Jacob Korevaar.
\newblock {\em Tauberian theory}, volume 329 of {\em Grundlehren der
  Mathematischen Wissenschaften [Fundamental Principles of Mathematical
  Sciences]}.
\newblock Springer-Verlag, Berlin, 2004.
\newblock A century of developments.

\bibitem[Lal87]{lalley}
S.~P. Lalley.
\newblock Distribution of periodic orbits of symbolic and {A}xiom {A} flows.
\newblock {\em Adv. in Appl. Math.}, 8(2):154--193, 1987.

\bibitem[LP82]{Lax}
Peter~D. Lax and Ralph~S. Phillips.
\newblock The asymptotic distribution of lattice points in {E}uclidean and
  non-{E}uclidean spaces.
\newblock {\em J. Functional Analysis}, 46(3):280--350, 1982.

\bibitem[Mar04]{Margulis}
Grigoriy~A. Margulis.
\newblock {\em On some aspects of the theory of {A}nosov systems}.
\newblock Springer Monographs in Mathematics. Springer-Verlag, Berlin, 2004.

\bibitem[Mir08]{Mirzakhani}
Maryam Mirzakhani.
\newblock Growth of the number of simple closed geodesics on hyperbolic
  surfaces.
\newblock {\em Ann. of Math. (2)}, 168(1):97--125, 2008.

\bibitem[MU03]{MU}
R.~Daniel Mauldin and Mariusz Urba\'{n}ski.
\newblock {\em Graph directed {M}arkov systems}, volume 148 of {\em Cambridge
  Tracts in Mathematics}.
\newblock Cambridge University Press, Cambridge, 2003.
\newblock Geometry and dynamics of limit sets.

\bibitem[Nus70]{Nussbaum}
Roger~D. Nussbaum.
\newblock The radius of the essential spectrum.
\newblock {\em Duke Math. J.}, 37:473--478, 1970.

\bibitem[OS12]{Shah}
Hee Oh and Nimish Shah.
\newblock The asymptotic distribution of circles in the orbits of {K}leinian
  groups.
\newblock {\em Invent. Math.}, 187(1):1--35, 2012.

\bibitem[Pat75]{Patterson}
S.~J. Patterson.
\newblock A lattice-point problem in hyperbolic space.
\newblock {\em Mathematika}, 22(1):81--88, 1975.

\bibitem[Pat76]{Pattersonn}
S.~J. Patterson.
\newblock The limit set of a {F}uchsian group.
\newblock {\em Acta Math.}, 136(3-4):241--273, 1976.

\bibitem[Pol84]{P-RPF}
Mark Pollicott.
\newblock A complex {R}uelle-{P}erron-{F}robenius theorem and two
  counterexamples.
\newblock {\em Ergodic Theory Dynam. Systems}, 4(1):135--146, 1984.

\bibitem[PP83]{PaPo}
William Parry and Mark Pollicott.
\newblock An analogue of the prime number theorem for closed orbits of {A}xiom
  {A} flows.
\newblock {\em Ann. of Math. (2)}, 118(3):573--591, 1983.

\bibitem[PU21]{UP}
Mark Pollicott and Mariusz Urba\'{n}ski.
\newblock Asymptotic counting in conformal dynamical systems.
\newblock {\em Mem. Amer. Math. Soc.}, 271(1327):v+139, 2021.

\bibitem[Rue76]{Ruelle}
D.~Ruelle.
\newblock Generalized zeta-functions for {A}xiom {A} basic sets.
\newblock {\em Bull. Amer. Math. Soc.}, 82(1):153--156, 1976.

\bibitem[Sar81]{sarnak}
Peter Sarnak.
\newblock Asymptotic behavior of periodic orbits of the horocycle flow and
  {E}isenstein series.
\newblock {\em Comm. Pure Appl. Math.}, 34(6):719--739, 1981.

\bibitem[Sel56]{Selberg1}
A.~Selberg.
\newblock Harmonic analysis and discontinuous groups in weakly symmetric
  {R}iemannian spaces with applications to {D}irichlet series.
\newblock {\em J. Indian Math. Soc. (N.S.)}, 20:47--87, 1956.

\bibitem[Sul79]{Sullivan}
Dennis Sullivan.
\newblock The density at infinity of a discrete group of hyperbolic motions.
\newblock {\em Inst. Hautes \'{E}tudes Sci. Publ. Math.}, 1(50):171--202, 1979.

\bibitem[URM21]{Munday}
Mariusz Urbański, Mario Roy, and Sara Munday.
\newblock {\em Ergodic Theory – Finite and Infinite, Thermodynamic Formalism,
  Symbolic Dynamics and Distance Expanding Maps}.
\newblock De Gruyter, 2021.

\bibitem[Wie33]{wiener}
N.~Wiener.
\newblock {\em Fourier Integral and Certian of Its Applications}.
\newblock Cambridge Mathematical Library, 1933.

\end{thebibliography}
\end{document}